\documentclass[review,onefignum,onetabnum]{siamart171218}
\usepackage{multirow}
\usepackage{amsmath}
\usepackage{color}
\usepackage{amssymb}
\usepackage{mathrsfs}

\newcommand\vvec{\mathbf{v}}

\newcommand\x{\mathbf{x}}
\newcommand\y{\mathbf{y}}
\newcommand\W{\mathbf{W}}
\newcommand\I{\mathbf{I}}
\newcommand\Hvec{\mathbf{H}}
\newtheorem{assumption}{Assumption}[section]

\newcommand\barxx{\bar{\mathbf{x}}_{\mathcal{V}}}
\newcommand\barxxx{\bar{\mathbf{x}}_{\mathcal{V}_\perp}}
\newcommand\xx{\mathbf{x}_{\mathcal{V}}}
\newcommand\xxx{\mathbf{x}_{\mathcal{V}{\hspace{-.05cm}\perp}}}

\usepackage{lipsum}
\usepackage{amsfonts}
\usepackage{graphicx}
\usepackage{epstopdf}
\usepackage{algorithmic}
\newsiamremark{remark}{Remark}
\newsiamremark{hypothesis}{Hypothesis}
\crefname{hypothesis}{Hypothesis}{Hypotheses}
\newsiamthm{claim}{Claim}

\definecolor{midbluepurple}{RGB}{134,0,102}

\newcommand{\TheTitle}{Stochastic Saddle Search With Convergence Guarantee}
\newcommand{\TheAuthors}{B. Shi, L. Zhang, Q. Du}

\headers{\TheTitle}{\TheAuthors}
\title{A Stochastic Algorithm for Searching Saddle Points with Convergence Guarantee\thanks{Submitted to the editors DATE.\funding{This work was partially supported by NSF DMS-2309245 and DOE DE-SC0025347, the National Key Research and Development Program of China  2024YFA0919500, and National Natural Science Foundation of China (No. 12225102, T2321001, 12288101, and 12426653).}
}}

\author{Baoming Shi\thanks{Department of Applied Physics and Applied Mathematics, Columbia University, New York, NY 10027, USA (\email{bs3705@columbia.edu}).}
\and Lei Zhang\thanks{Beijing International Center for Mathematical Research, Center for Quantitative Biology, Center for Machine Learning Research, Peking University, Beijing 100871, China (\email{zhangl@math.pku.edu.cn}).}
\and Qiang Du\thanks{Corresponding author. Department of Applied Physics and Applied Mathematics, and Data Science Institute, Columbia University, New York, NY 10027, USA (\email{qd2125@columbia.edu}).}
}
\ifpdf
\hypersetup{
  pdftitle={\TheTitle},
  pdfauthor={\TheAuthors}
}
\fi
\renewcommand{\TheTitle}{A Stochastic Algorithm for Searching Saddle Points with Convergence Guarantee}

\begin{document}
\maketitle

\begin{abstract}
Saddle points provide a hierarchical view of the energy landscape, revealing transition pathways and interconnected basins of attraction, and offering insight into the global structure, metastability, and possible collective mechanisms of the underlying system. In this work, we propose a stochastic saddle-search algorithm to circumvent exact derivative and Hessian evaluations that have been used in implementing traditional and deterministic saddle dynamics. At each iteration, the algorithm uses a stochastic eigenvector-search method, based on a stochastic Hessian, to approximate the unstable directions, followed by a stochastic gradient update with reflections in the approximate unstable direction to advance toward the saddle point. We carry out rigorous numerical analysis to establish the almost sure convergence for the stochastic eigenvector search and local almost sure convergence with an $O(1/n)$ rate for the saddle search, and present a theoretical guarantee to ensure the high‑probability identification of the saddle point when the initial point is sufficiently close. Numerical experiments, including the application to a neural network loss landscape and a Landau-de Gennes type model for nematic liquid crystal, demonstrate the practical applicability and the ability for escaping from ``bad" areas of the algorithm.
 \end{abstract}



\section{Introduction}
\label{sec:intro}
Studying energy landscapes provides a powerful lens across disciplines, ranging from materials science and chemistry to biology and engineering, by uncovering the mechanisms that drive structural transformations, dynamical evolutions, and emergent behaviors in complex systems \cite{hanggi1990reaction}.
A critical component is the investigation of saddle points, which offers insight to understanding important scientific questions such as the formation of the critical nuclei and transition pathways in phase transformations \cite{zhang2007morphology, yin2021transition, LIN20101797,zhang2025solution}, the defect configurations in liquid crystals \cite{robinson2017molecular,shi2024multistability,dalby2024stochastic,shi2025modified}, chemical reactions \cite{kraka2010computational,nagahata2013reactivity}, protein folding \cite{burton1997energy}, excited states in Bose-Einstein condensates \cite{bao2003ground, yin2024revealing,liu2023constrained}, and the loss landscape in deep neural networks \cite{abbe2023sgd,fukumizu2019semi,zhang2021embedding}. We refer to \cite{henkelman2000methods,vanden2010transition,zhang2016recent} for additional reviews.

The computation of saddle points is more challenging than that of stable minima, as the unstable directions of a saddle point are not known a priori, and the unstable directions introduce instability into the system. The existing numerical algorithms for finding saddle points can be generally divided into two classes: path-finding methods and surface-walking methods. The former includes the string method \cite{weinan2002string,weinan2007simplified}, which is used to search for the minimum energy path (MEP), and the points along the MEP with locally maximum energy value are the saddle points. The latter methods include the gentlest ascent dynamics \cite{gad}, the dimer method \cite{henkelman1999dimer,zhangdu2012,zhang2016optimization}, the min-max method \cite{li2001minimax}, and the high-index saddle dynamics (HiSD) \cite{2019High, JJIAM2023,su2025improved}, which evolve an iteration on the potential energy surface toward a saddle point and enables us to construct a pathway map of connected critical points \cite{yin2020construction,yin2021searching}. 

However, the HiSD method requires evaluating gradients and Hessians of the objective function at each step, which are often either unavailable or computationally expensive. Typical examples include:
(i) the free energy surface in molecular simulations, which is not an explicit function of atomic coordinates but rather a statistical average over microscopic states \cite{mezei1986free};
(ii) the loss function of neural networks, which is computationally expensive when the dataset is extremely large \cite{goodfellow2016deep};
(iii) cases where the objective function is an energy functional, and its \text{Fr\'echet} derivative is defined through a partial differential equation. To improve computational efficiency, various methods have been proposed that bypass exact evaluation of gradients or Hessians, such as the Gaussian process-based saddle dynamics \cite{JJIAM2023, gu2022active} and neural network-based saddle dynamics \cite{liu2024neural}. The main idea of these approaches is to construct computationally efficient surrogate gradients or stochastic approximations of the exact gradients. While these methods have shown promising empirical results on certain benchmark problems, a rigorous theoretical analysis establishing their convergence is still lacking, which is crucial for practical applications.

There have been a number of numerical analysis works of the (discretized) saddle dynamics and its variants, such as convergence analysis \cite{zhangdu2012,levitt2017convergence,luo2022sinum,gould2016dimer} and error estimates \cite{erroreulersd,li2025error, miao2025construction}. However, very few studies have focused on the analysis of their stochastic variants. While stochastic gradient descent (SGD) \cite{ruder2016overview} has been extensively studied for locating minimizers and is hugely successful in many applications, there is relatively little work on the development and the mathematical analysis of stochastic algorithms for the computation of saddle points with unknown unstable direction. Compared to SGD, which primarily focuses on updating iterates toward a minimum, stochastic saddle-search algorithms should also account for the update of unstable directions that are unknowns associated with the iterates in the configuration space. This introduces significant challenges in algorithm design and convergence analysis. The main contribution of this work is to, for the first time to our knowledge, provide a stochastic saddle-search algorithm with convergence guarantee, which offers a practical computational tool and a rigorous numerical analysis framework applicable to a wide class of stochastic or surrogate model-based saddle dynamics.

The paper is organized as follows. In \Cref{sec: Saddle point and saddle dynamics}, we review the concept of saddle points and the saddle dynamics. \Cref{sec: determined unstable space} presents a stochastic saddle-search algorithm under the assumption of a known unstable space and convex-concave structure, and establishes its global almost sure convergence. In \Cref{sec: undetermined unstable space}, we consider the more challenging case where the unstable space is unknown. We first assume that the exact unstable eigenvector is available at each iteration and establish local almost sure convergence under this assumption. For the general case, the algorithm employs a stochastic eigenvector-search procedure, based on a stochastic approximation of the Hessian, to estimate the unstable directions at each iteration. With these approximate unstable directions, we establish the local high-probability convergence of the stochastic saddle-search algorithm, with a convergence rate of $O(1/n)$ with $n$ denoting the number of iterations performed. \Cref{sec: numerical results} presents numerical experiments on several practical problems to substantiate the theoretical results. We finally present our conclusion and discussion in \Cref{sec: conclusion}.

\section{Saddle point and saddle dynamics}\label{sec: Saddle point and saddle dynamics}

Mathematically, the energy landscape is a mapping $ f(\x):\mathbb{R}^d\rightarrow \mathbb{R}
$ of configurations of the system to their energies.
We assume that $f$ is $C^3$ continuous. A critical (stationary) point is a point where the derivative of $f$ is zero, i.e.,  $\nabla f(\x)=0$. A simple criterion for checking the stability of a critical point $\x^*$ is to compute its Hessian matrix $\nabla^2 f(\x^*)$. If $\nabla^2 f(\x^*)$ is positive definite, then $\x^*$ is a local minimum \cite{nocedal1999numerical}; If $\nabla^2 f(\x^*)$ has both positive and 
$k$ negative eigenvalues with the corresponding $k$ eigenvectors $\{\vvec_1^*,\cdots,\vvec_k^* \big \}$ referred to as ``unstable eigenvectors" or ``unstable directions" in this paper, then $\x^*$ is a index-$k$ saddle point. By setting the $k$-dimensional subspace $\mathcal{V}=\text{span} \big \{\vvec_1^*,\cdots,\vvec_k^* \big \}$, $\x^*$ is a local maximum on $\x^*+\mathcal{V}$ and a local minimum on $\x^*+\mathcal{V}^\perp$, where $\mathcal{V}^\perp$ is the orthogonal complement of $\mathcal{V}$. 

The saddle dynamics for searching an index-$k$ saddle point $\x$ is defined by \cite{2019High}
\begin{equation}
    \left\{
    \begin{aligned}
		\dot{\x}&=- \left(\I-2\sum_{i=1}^k {\vvec}_i{\vvec}_i^\top\right)\nabla f(\x), \\
		\dot{\vvec}_i&=-   \left(\I-{\vvec}_i{\vvec}_i^\top-\sum_{i=1}^{i-1}2{\vvec}_j{\vvec}_j^\top\right)\nabla^2 f(\x) \vvec_i,\ i=1,\cdots,k ,\\
    \end{aligned}
    \right.
	\label{eq: SD}
\end{equation}
where $1\leqslant k\leqslant d-1$ and $\I$ is the identity operator. The dynamics for $\x$ in \eqref{eq: SD} can be written as
$\dot{\x}
    = \left( \I-\mathcal{P}_{\mathcal{V}}  \right)\left(-\nabla f(\x)\right)+ \mathcal{P}_{\mathcal{V}} \left(\nabla f(\x)\right)$, where $\mathcal{P}_{\mathcal{V}}=\sum_{i=1}^k {\vvec}_i{\vvec}_i^\top$ is the orthogonal projection on $\mathcal{V}$. Thus, $\left( \I-\mathcal{P}_{\mathcal{V}}  \right)\left(-\nabla f(\x)\right)$ is a descent direction on $\mathcal{V}^\perp$, and $\mathcal{P}_{\mathcal{V}} \left(\nabla f(\x)\right)$ is an ascent direction on $\mathcal{V}$. 

The continuous dynamics for $\vvec_i, i=1,2,\cdots,k$ in \eqref{eq: SD} can be obtained by minimizing the $k$ Rayleigh quotients simultaneously with the gradient type dynamics,
$$
	\min_{{\vvec}_i}\text{  }\left<\vvec_i, \nabla ^2 f(\x)\vvec_i\right>,\ \text{s.t.}\ \left<\vvec_i,\vvec_j\right>=\delta_{ij}, \ j=1,2,\cdots,i,
$$
which generates the subspace $\mathcal{V}$ by computing the eigenvectors corresponding to the smallest $k$ eigenvalues of $\nabla^2 E (\x)$ (counting the multiplicity).

The discretized version of the above saddle dynamics can be used to locate a saddle point of
any index. However, it requires evaluating both the gradient and Hessian of the objective function at each step, which might be infeasible in practice or computationally expensive in high dimensions. We consider a stochastic saddle-search algorithm that circumvents exact derivative (and Hessian) evaluations by employing stochastic approximations.

\section{Saddle-search with a known unstable space and the stochastic variant}\label{sec: determined unstable space}
We first consider a simple scenario where the unstable directions $\vvec_i$, $i=1,\cdots,k$ at the target saddle point $\x^*$ are known a priori, and the problem becomes
\begin{equation}
\min_{\xxx \in\mathcal{V}^\perp}\max_{\xx \in\mathcal{V}} f(\xx+\xxx),
\label{simple scenario}
\end{equation}
where $\mathcal{V}=\text{span}\{\vvec_j\}_{j=1}^k$ with $\mathcal{V}^\perp$ the orthogonal complement, and $\x=\xx+\xxx=\mathcal{P}_\mathcal{V}\x+(\I-\mathcal{P}_\mathcal{V})\x$. The solution of \eqref{simple scenario} is a saddle point characterized by 
$$
f(\xxx^*+\xx)\leqslant f(\xxx^*+\xx^*)\leqslant f(\xxx+\xx^*).
$$

An example of this simple scenario with prescribed stable and unstable subspaces is the saddle point of the Lagrange function $\mathcal{L}(\x,\nu)=f(\x)+\sum_{i=1}^{d_\nu} \nu_i h_i(\x)$
for an equality-constrained optimization problem with $f$ being the objective function and $\{h_i (\x)= 0\} $ being the constraints. With $\{\nu_i\}$ being the Lagrange multipliers, we have
$$ 
\mathcal{L}(\x^*,\mathbf{\nu})\leqslant \mathcal{L}(\x^*,\mathbf{\nu}^*)\leqslant \mathcal{L}(\x,\mathbf{\nu}^*),\quad\forall  \x\in \mathbb{R}^d, \mathbf{\nu} \in \mathbb{R}^{d_\nu},$$
where the unstable space is the definition domain of the multipliers, $\mathbf{\nu}$, and the stable space is the definition domain of $\x$, i.e., $\mathcal{V}=\mathbb{R}^{d_\nu},\mathcal{V}^\perp= \mathbb{R}^d$.

For the saddle point problem \eqref{simple scenario}, we can make the following assumption on the objective function $f=f(\x)$ to assure its well-posedness and non-degeneracy.

\begin{assumption}[Convex-concave]\label{convex-concave} 
The objective function $f(\x)=f(\xx+\xxx)$ is strongly convex with respect to $\xxx$ and strongly concave with respect to $\xx$. 
\end{assumption}

For example, for a strongly convex
$f$ with affine constraints, the regularized Lagrangian function $\tilde{\mathcal{L}}(x,\nu;\eta)=\mathcal{L}(x,\nu)-\frac{\eta}{2}\Vert\nu\Vert_2^2$ \cite{badiei2016distributed} is strongly convex-concave. For strongly convex-concave functions, it is well known that the saddle point exists and is unique \cite{nocedal1999numerical}.

With unstable and stable spaces known, the Saddle Dynamics,
or the Gradient-Ascent-Descent Dynamics, $\hat{\x}=\hat{\x}(t)$ associated with \eqref{simple scenario} becomes, for $t>0$, 
\begin{equation}
    \left\{
    \begin{aligned}
	&	\frac{d\hat{\xx}}{dt}
= \mathcal{P}_\mathcal{V} \nabla f(\hat{\x}), \\
&\frac{d\hat{\xxx}}{dt}
        = -(\I-\mathcal{P}_\mathcal{V})\nabla f(\hat{\x}),
    \end{aligned}
    \right.
	\label{eq: simple SD}
\end{equation}
together with an initial condition $
(\hat{\xx}(0),\hat{\xxx}(0))$. In this case, we easily see that the unique saddle point $(\xx^*,\xxx^*)$ is an asymptotically stable equilibrium of \eqref{eq: simple SD}.

\begin{proposition}\label{Lyapunov}
Suppose \Cref{convex-concave} holds, the stationary point of \eqref{eq: simple SD} 
is Lyapunov stable. Moreover, with any initial condition, the solution $(\hat{\xx}(t),\hat{\xxx}(t))$ of \eqref{eq: simple SD} asymptotically converges to the saddle point $(\xx^*,\xxx^*)$
as $t\to \infty$.
\end{proposition}
\begin{proof}
Let $\|\cdot\|_2$ denote the Euclidean norm. Define the Lyapunov function as
\begin{equation}
    D(\hat{\xx},\hat{\xxx})=\frac{1}{2}(\Vert \hat{\xx}-\xx^*\Vert_2^2+\Vert \hat{\xxx}-\xxx^* \Vert^2_2) \geqslant 0.
    \label{Distance}
\end{equation}
We can directly obtain that $
    D=0$ if and only if $  (\hat{\xx},\hat{\xxx})=(\xx^*,\xxx^*)
$. Moreover,
if $(\hat{\xx},\hat{\xxx})\neq (\xx^*,\xxx^*)$, 
$$
    \begin{aligned}
    \dot{D}&=\langle \nabla_{\hat{\xx}} D,\mathcal{P}_{\mathcal{V}} \nabla f(\hat{\x})\rangle+\langle\nabla_{\hat{\xxx}} D,-(I-\mathcal{P}_{\mathcal{V}}) \nabla f(\hat{\x})\rangle\\
    &=\langle \hat{\xx}-\xx^*,\mathcal{P}_{\mathcal{V}} \nabla f(\hat{\x})\rangle+\langle \hat{\xxx}-\xxx^*,-(I-\mathcal{P}_{\mathcal{V}}) \nabla f(\hat{\x})\rangle\\
    &< f(\xx^*+\hat{\xxx})-f(\hat{\xx}+\hat{\xxx})+f(\hat{\xx}+\hat{\xxx})-f(\hat{\xx}+\xxx^*)\\
    &= f(\xx^*+\hat{\xxx})-f(\hat{\xx}+\xxx^*)< 0.
    \end{aligned}
$$
Then the saddle point $(\xx^*,\xxx^*)$ is an asymptotically stable equilibrium of \eqref{eq: simple SD}.
\end{proof}

\textbf{Stochastic saddle-search algorithm with a known unstable space}: With $\x(0)=\xx(0)+\xxx(0)$, we consider a discrete and stochastic version of the continuous saddle dynamics, so that $\x(n)=\xx(n)+\xxx(n)$ is updated from
\begin{equation}
\label{eq: stochastic iteration}
\left\{
    \begin{aligned}
&\xx(n+1)=\xx(n)+\alpha(n)\mathcal{P}_\mathcal{V}\nabla f(\xx(n)+\xxx(n);\omega(n)), \\
&\xxx(n+1)=\xxx(n)-\alpha(n)(\I-\mathcal{P}_\mathcal{V})\nabla f(\xx(n)+\xxx(n);\omega(n)),
    \end{aligned}\right.
\end{equation}
where $\nabla f(\x;\omega)$ is a stochastic gradient approximation of $\nabla f(\x)$ with $\omega$ drawn from some probability space $(\Omega, \mathcal{F}, \mathbb{P})$, and $\alpha(n)>0$ is the step size satisfying
\begin{assumption}[Decay step size]\label{step size sec 1}
$\displaystyle
\sum_{n=0}^\infty \alpha(n)=\infty$, and  $\,\displaystyle \sum_{n=0}^\infty \alpha(n)^2<\infty$.
\end{assumption}

We also make the following 
assumptions for all $\x\in \mathbb{R}^d$ and constants $G$ and $\sigma$.
\begin{assumption}[Regularity]\label{regularity}
    $\nabla f$ is Lipschitz continuous, and $\Vert\nabla f(\x)
    \Vert_2\leqslant G$.
\end{assumption}

\begin{assumption}[Unbiased]\label{unbias} The error term has zero mean and finite second moments, i.e., $\mathbb{E}(\nabla{f(\x,\omega)})=\nabla f(\x), \mathbb{E}(\Vert\nabla{f(\x,\omega)-\nabla f(\x)\Vert^2_2})\leqslant \sigma^2$,
$\forall \x\in \mathbb{R}^d$.
\end{assumption}

The continuous-time interpolated trajectories $(\barxx(t),\barxxx(t))$ is defined as
\begin{equation}
\label{interpolated trajectory}
 \left\{   \begin{aligned}
\barxx(t)&=\xx(n)+(\xx(n+1)-\xx(n))\frac{t-t(n)}{\alpha(n)},\; t\in [t(n),t(n+1)],\\
\barxxx(t)&=\xxx(n)+(\xxx(n+1)-\xxx(n))\frac{t-t(n)}{\alpha(n)},\;t\in [t(n),t(n+1)],
\end{aligned}\right.
\end{equation}
where $\xx(n)$ and $\xxx(n)$ are generated from \eqref{eq: stochastic iteration}, and $t(n+1)=\sum_{i=0}^n \alpha(n).$

In \cite{benaim1996asymptotic}, the authors introduced an asymptotic pseudo-trajectory of
a continuous dynamic system as a trajectory that approximates the true trajectory of the dynamics as time goes to $ \infty$ over any finite time window. We state the definition corresponding to the
dynamics \eqref{eq: simple SD}
below for easy reference.

\begin{definition}
\label{def-quasi}
$\bar{x}(t)=(\barxx(t),\barxxx(t))$
is called
an asymptotic pseudo-trajectory of 
the dynamic system 
\eqref{eq: simple SD}, if
there is a solution $(\hat{\xx}(t),\hat{\xxx}(t))$ of  \eqref{eq: simple SD} with initial condition $(\hat{\xx}(0)=\barxx(t),\hat{\xxx}(
0
)=\barxxx(t))$, almost surely such that, $\forall T>0$,
$
\lim_{t\rightarrow \infty} \sup_{0\leqslant t_1 \leqslant T}\Vert \bar{\x}(t+t_1)-\hat{\x}(t_1)\Vert_2=0.
$
\end{definition}

\begin{lemma}\label{pseudo-trajectory}
With \Cref{convex-concave}, \Cref{step size sec 1}, \Cref{regularity}, and \Cref{unbias}, if the interpolated trajectory is bounded with probability one, the interpolated trajectory $(\barxx(t),\barxxx(t))$ is an asymptotic pseudo-trajectory of the dynamics \eqref{eq: simple SD}, as specified by the Definition \ref{def-quasi}.
\end{lemma}
\begin{proof}
    It is an immediate corollary of  \cite[Proposition 5.1]{benaim1996asymptotic}.
\end{proof}

In fact, the boundedness assumption in \Cref{pseudo-trajectory} follows from \Cref{convex-concave}, \Cref{step size sec 1}, \Cref{regularity}, and \Cref{unbias}.

\begin{proposition}[Boundedness]\label{proposition: boundedness}
  Under \Cref{convex-concave}, \Cref{step size sec 1}, \Cref{regularity}, and \Cref{unbias}, the sequence generated by the stochastic saddle-search algorithm is almost surely bounded, i.e., there exists a constant $C$ such that  $\sup_n \Vert \xx(n)+\xxx(n)\Vert_2 \leqslant C <\infty$
    almost surely, i.e., with probability 1.
\end{proposition}
\begin{proof}
Note that $\Vert(\I-\sum_{i=1}^k2 \vvec_i\vvec_i^\top)\nabla f(\x,\omega) \Vert_2=\Vert \nabla f(\x,\omega) \Vert_2$,
    \begin{equation*}
    \begin{aligned}
    & \Vert\x(n+1)-\x^*\Vert^2_2
        =\Vert \x(n)-\x^* \Vert^2_2+\alpha(n)^2\Vert \nabla f(\x(n);\omega(n)) \Vert^2_2\\  &\; +2\alpha_n(\langle \xx(n)-\xx^*,\nabla_{\xx}f(\x(n);\omega(n)) \rangle-\langle \xxx(n)-\xxx^*,\nabla_{\xxx}f(\x(n);\omega(n)) \rangle).
    \end{aligned}
    \end{equation*}
Thus, with filtration $
  \mathcal{F}(n) = \{\x_0, \x_1, \dots, \x_n\}
  $, we have
\begin{equation*}
\begin{aligned}
    \mathbb{E}&(\Vert\x(n+1)-\x^*\Vert^2_2|\mathcal{F}(n))\\
    \leqslant&\Vert \x(n)-\x^* \Vert^2_2+2\alpha(n)^2(\sigma^2+G^2)\\
    &+2\alpha_n(\langle \xx(n)-\xx^*,\nabla_{\xx}f(\x(n)) \rangle-\langle \xxx(n)-\xxx^*,\nabla_{\xxx}f(\x(n)) \rangle)\\
    \leqslant&\Vert \x(n)-\x^* \Vert^2_2+2\alpha(n)^2(\sigma^2+G^2)\\
    &+2\alpha(n)\underbrace{f(\xx^*,\xxx)-f(\xx,\xxx)+f(\xx,\xxx)-f(\xx,\xxx^*)}_{\leqslant 0, \text{\Cref{convex-concave}}}\\
    \leqslant &\Vert \x(n)-\x^* \Vert^2_2+ 2\alpha(n)^2(\sigma^2+G^2).
\end{aligned}
\end{equation*}
With $\zeta_n:=\Vert \x(n)-\x^* \Vert^2+\sum_n^\infty 2\alpha(n)^2(\sigma^2+G^2)$ and \Cref{step size sec 1}, we see that
 $$   \mathbb{E}\left(\zeta_{n+1}
 |\mathcal{F}(n)\right)
\leqslant \zeta_n.$$
This
demonstrates that $\{\zeta_n\}$ is a non-negative supermartingale \cite{williams1991probability}, and converges almost surely to a random variable, $\zeta_\infty$, with a finite expectation (Doob's martingale convergence theorems). As $n\rightarrow \infty$, $
    \Vert \x(n)-\x^* \Vert_2^2\leqslant \zeta_n \rightarrow  \zeta_\infty
    <\infty$ almost surely,
and the proposition is thus proven.
\end{proof}

As a consequence of the asymptotic convergence of the continuous saddle dynamics to a saddle point established in \Cref{Lyapunov}, we get the convergence of the interpolated trajectories.
\begin{proposition}\label{proposition: converge with determined space}
    The iterations generated by the stochastic saddle-search algorithm with \Cref{convex-concave}, \Cref{step size sec 1}, \Cref{regularity}, and \Cref{unbias}, converge to a saddle point almost surely. 
\end{proposition}

\begin{proof}
    The argument employs an approach in which the true trajectory guides the pseudo-trajectory toward convergence \cite{benaim2006dynamics, borkar2008stochastic} (see details in the Supplementary Materials). 
\end{proof}

\section{Stochastic algorithm with an unknown unstable space}\label{sec: undetermined unstable space}
We now turn to a more general and practically relevant scenario, where neither the location of the saddle point nor its unstable directions is known. While searching for the saddle point, the algorithm should simultaneously find its associated unstable space on the fly, which introduces more challenges.

\subsection{Exact unstable direction at each iteration point}
We first assume that, at each iteration $\x(n)$, the exact eigenvectors corresponding to the $k$ smallest eigenvalues of $\nabla^2f(\x(n))$ can be obtained (see later discussions in \Cref{sec: Inexact eigenvector} on cases with inexact eigenvectors). Then, the saddle dynamics is defined as
\begin{equation}
\label{eq: exact saddle dynamics}
	\dot{\x}=- \underbrace{\left(\I-2\sum_{i=1}^k {\vvec}_i(\x){\vvec}_i(\x)^\top\right)}_{\mathcal{P}_{V(\x)}}\nabla f(\x)
\end{equation}
where $\vvec_i(\x),i=1,\cdots,k$ are the exact eigenvectors corresponding to the $k$ smallest eigenvalues of $\nabla^2f(\x)$, and the stochastic algorithm is 
$$
    \x(n+1)=\x(n)-\alpha(n) \mathcal{P}_{V(\x(n))}\nabla f(\x(n);\omega(n)).
$$
For the brevity of the notation, we use $\mathcal{P}_V$ to represent $\mathcal{P}_{V(\x(n))}$.

For the general objective function, we give the following assumption: 
\begin{assumption}\label{assumption: Hessian}
There exist $\delta>0, L>\mu>0$, if $\x\in U=\{\x,\Vert \x-\x^* \Vert_2\leqslant \delta\}$, then the eigenvalues of $\nabla^2 f(\x)$ satisfy
\begin{equation}
    -L<\lambda_1\leqslant\cdots\leqslant\lambda_k< -\mu < 0< \mu<\lambda_{k+1} \leqslant \cdots \leqslant \lambda_d < L,
    \label{eq:eigenvalues}
\end{equation}
and the Lipschitz condition $\Vert \nabla^2 f(\x)-\nabla^2f(\mathbf{y}) \Vert_2 \leqslant M\Vert \x-\mathbf{y} \Vert_2$ holds for $\forall\x,\mathbf{y}\in U$ with a constant $M>0$.
\end{assumption}
\begin{remark}
    Since $f\in C^3$, it is locally homeomorphic to $\tilde{f}(\x)=-x_1^2-x_2^2-\cdots-x_k^2+x_{k+1}^2+\cdots+x_d^2$ near a non-degenerate index-$k$ saddle point $\x^*$. Thus, the eigenvalue assumption holds for every strict saddle point $\x^*$.
\end{remark}

\begin{proposition}\label{descent with exact eigenvector}
Under \Cref{assumption: Hessian}, if $\alpha(n)< \frac{1}{2\mu}$, and
 $\x(n)\in U$ with $U=\{\x,\Vert \x-\x^* \Vert_2 \leqslant 
 \min ( \frac{\mu}{M},\delta)\}$, we have
 $$
        \Vert\x(n+1)-\x^*\Vert^2_2\leqslant (1-
        \alpha(n) \mu)\Vert\x(n)-\x^*\Vert^2_2+\alpha(n)\psi(n)+\alpha(n)^2\Vert\nabla f(\x(n);\omega(n))\Vert^2_2,
$$
where $\psi(n) =-\langle 2\mathcal{P}_{V}(\nabla f(\x(n);\omega(n))-\nabla f (\x(n) ), \x(n) -\x^* \rangle$ is a martingale sequence with zero expectation.
\end{proposition}

\begin{proof}
By definition,
\begin{equation}
\begin{aligned}       &\Vert\x(n+1)-\x^*\Vert^2_2
        =\Vert \x(n)-\x^*-\alpha(n)\mathcal{P}_{V}\nabla f(\x(n);\omega(n)) \Vert_2^2\\
        &=\Vert \x(n)-\x^* \Vert^2_2-2\alpha(n)\langle \mathcal{P}_{V}\nabla f(\x(n);\omega(n)),\x(n)-\x^*  \rangle\\
        &\quad \quad+\alpha(n)^2\Vert \mathcal{P}_{V} \nabla f(\x(n);\omega(n)) \Vert^2_2\\
        &=\Vert \x(n)-\x^* \Vert^2_2-2\alpha(n)\langle \mathcal{P}_{V}\nabla f(\x(n)),\x(n)-\x^*  \rangle+\alpha(n)^2\Vert \mathcal{P}_{V} \nabla f(\x(n);\omega(n)) \Vert^2_2\\
        &\quad \quad -\alpha(n)\langle 2\mathcal{P}_{V}(\nabla f (\x(n);\omega(n))-\nabla f (\x(n))),\x(n)-\x^* \rangle\,.
\end{aligned}\label{expansion exact}
\end{equation}
Note that 
\begin{equation}
\begin{aligned}   &\langle\mathcal{P}_{V}\nabla f(\x(n)),\x(n)-\x^*\rangle =\left\langle\mathcal{P}_{V}(\nabla f(\x(n))-\nabla f(\x^*)),\x(n)-\x^*\right\rangle\\
&=\left\langle\mathcal{P}_{V}\int_0^1 \nabla^2 f(\x^*+s(\x(n)-\x^*))\mathrm{d} s(\x(n)-\x^*),\x(n)-\x^*\right\rangle\\
&=\langle\mathcal{P}_{V} \nabla^2 f(\x(n))(\x(n)-\x^*),\x(n)-\x^*\rangle \\
&\quad \quad +\left\langle\mathcal{P}_{V}\int_0^1 (\nabla^2 f(\x^*+s(\x(n)-\x^*))-\nabla^2 f(\x(n)) \mathrm{d} s(\x(n)-\x^*),\x(n)-\x^*\right\rangle\\
&\geqslant \langle\mathcal{P}_{V} \nabla^2 f(\x(n))(\x(n)-\x^*),\x(n)-\x^*\rangle-\frac{M}{2}\Vert \x(n)-\x^* \Vert^3_2.\\
    &\geqslant \left(\mu-\frac{M\Vert \x(n)-\x^* \Vert_2}{2}\right)\Vert \x(n)-\x^* \Vert^2_2\geqslant\frac{\mu}{2}\Vert \x(n)-\x^* \Vert^2_2.
    \end{aligned}
    \label{descent part exact}
\end{equation}

The second-to-last inequality in \eqref{descent part exact} holds
because, by \eqref{eq:eigenvalues}
in the \Cref{assumption: Hessian}, $\mathcal{P}_{V} \nabla^2 f(\x(n))$ is positive definite 
with its smallest eigenvalue larger than $\mu$. By substituting \eqref{descent part exact} into \eqref{expansion exact}, we have that
$$
        \begin{aligned}
        &\Vert\x(n+1)-\x^*\Vert^2_2\\
        &=\Vert \x(n)-\x^* \Vert^2_2-2\alpha(n)\langle \mathcal{P}_{V}\nabla f(\x(n)),\x(n)-\x^*  \rangle+\alpha(n)^2\Vert \mathcal{P}_{V} \nabla f(\x(n);\omega(n)) \Vert^2_2\\
        &\quad \quad -2\alpha(n)\langle \mathcal{P}_{V}(\nabla f (\x(n);\omega(n))-\nabla f (\x(n))),\x(n)-\x^* \rangle\\
        &\leqslant (1-\alpha(n) \mu)\Vert\x(n)-\x^*\Vert^2_2+\alpha(n)\psi(n)+\alpha(n)^2\Vert\nabla f(\x(n);\omega(n))\Vert^2_2.
        \end{aligned}
$$
This proves the proposition.
\end{proof}

\begin{lemma}\label{lemma: An and Bn}
Let $A(\x): \mathbb{R}^d \to \mathbb{R}$ be a function bounded from below, $B(\x): \mathbb{R}^d \to \mathbb{R}$ be a non-negative function, i.e., $B(\x) \geqslant 0$ for all $\x \in \mathbb{R}^d$, $\mathcal{F}(n)$ denote the filtration generated by the random variables
  $
  \mathcal{F}(n) = \{\x_0, \x_1, \dots, \x_n\}.
  $
With step size assumption in \Cref{step size sec 1}, suppose there exist constants $C_2, C_3 > 0$ such that 
\begin{equation}
\label{key equation for convergence}
    \mathbb{E}(A(\x(n+1))-A(\x(n))|\mathcal{F}(n))\leqslant-C_2\alpha(n)B(\x(n))+C_3\alpha(n)^2,
\end{equation}
then $A(\x(n))$ almost surely converges as $n\rightarrow \infty$. If $B(\x(n))$ also converges almost surely, then it almost surely converges to zero.
\end{lemma}

\begin{proof}
When $A(\x)=B(\x)$, it coincides with the Robbins-Siegmund theorem \cite{robbins1971convergence}. The proof of the lemma is provided in the Supplementary Materials.
\end{proof}

\begin{proposition}\label{local convergence with exact eigenvector}
    Suppose the event that iterate sequence ${\x(n)}$ lies within the set $U=\left\{\x,\Vert \x-\x^* \Vert_2\leqslant \min\left(\frac{\mu}{M},\delta\right) \right\}$ occurs with a positive probability. Conditioned on this event, the iterate sequence generated by the stochastic saddle-search algorithm with \Cref{step size sec 1}, \Cref{unbias}, and \Cref{assumption: Hessian} converges to the index-$k$ saddle point $\x^*$ almost surely.
\end{proposition}
\begin{proof}
Define the events $E^n = \{ \x(i) \in U, i=1,2,\cdots,n \}$ and corresponding indicator function $\mathbf{1}_{E^n}$. Set $G=\max_{\x\in U} \Vert\nabla f(\x)\Vert_2$, from \Cref{descent with exact eigenvector}, we have
$$
\begin{aligned}
&\mathbb{E}\left( \mathbf{1}_{E^{n+1}}\Vert \x(n+1) - \x^*\Vert_2^2 - \mathbf{1}_{E^n}\Vert\x(n) - \x^*\Vert_2^2 |\mathcal{F}(n) \right) \\
&\leqslant \mathbb{E}\left( \mathbf{1}_{E^{n}}\Vert \x(n+1) - \x^*\Vert_2^2 - \mathbf{1}_{E^n}\Vert\x(n) - \x^*\Vert_2^2 |\mathcal{F}(n) \right) \\
&\leqslant -\alpha(n) \mu \mathbf{1}_{E^n} \Vert\x(n) - \x^*\Vert_2^2 + 2\alpha(n)^2 (\sigma^2+G^2).
\end{aligned}
$$
By setting
$
A(\x(n)) = B(\x(n)) = \mathbf{1}_{E^n}\Vert\x(n) - \x^*\Vert_2^2$, $C_2 = \mu$, $C_3 = 2(\sigma^2+G^2)$,
and applying \Cref{lemma: An and Bn}, we conclude that $\mathbf{1}_{E^\infty}\Vert\x(n) - \x^*\Vert_2^2 \leqslant \mathbf{1}_{E^n} \Vert\x(n) - \x^*\Vert_2^2$ converges almost surely to zero.

\end{proof}

\begin{remark}
    In \Cref{attraction domain}, we will demonstrate that the boundedness event happens with a large probability if the initial condition $\x(0)$ is sufficiently close to $\x^*$ and the step size is small enough.
\end{remark}

For an objective function whose Hessian enjoys a special structure globally, the global convergence of the stochastic saddle-search algorithm can be achieved. This is stated in the following proposition.

\begin{proposition}[Global convergence with slightly changed Hessian] \label{Global convergence with exact eigenvector}
    Let the eigenvalues of $\nabla^2 f(\x)$ satisfy
    \eqref{eq:eigenvalues}
and $\Vert\nabla^2 f(\x)-\nabla^2 f(\mathbf{y})\Vert_2\leqslant M_1<\mu$, $\forall \x, \mathbf{y} \in \mathbb{R}^d$, then $f$ has a unique critical point $\x^*$ which is an index-k saddle point. With any initial condition, the stochastic saddle-search algorithm with \Cref{step size sec 1} and \Cref{unbias} almost surely converges to $\x^*$.
\end{proposition}

\begin{proof}
We first prove the existence and uniqueness of the critical point. Since
$$
\begin{aligned}
    \Vert\nabla f(\x)-\nabla^2 f(\mathbf{0})\x\Vert_2&=\left\Vert\nabla f(\mathbf{0})+\int_0^1 (\nabla^2f(s\x)-\nabla^2f(\mathbf{0}))\x\mathrm{d}s\right\Vert_2\\
    &\leqslant\Vert\nabla f(\mathbf{0})\Vert_2+M_1\Vert\x\Vert_2, \quad \forall \x \in \mathbb{R}^d,
\end{aligned}
$$
we get$\Vert\nabla f(\x)\Vert_2\geqslant (\mu-M_1)\Vert\x\Vert_2-\Vert\nabla f(\mathbf{0})\Vert_2\rightarrow \infty$, as $\Vert\x\Vert_2 \rightarrow \infty$. Together with the fact that the Jacobian of $\nabla f$, i.e., the Hessian  $\nabla^2 f$, is invertible, it follows that $\nabla f$ is a proper mapping. By Hadamard's global inverse function theorem   \cite{krantz2002implicit}, $\nabla f: \mathbb{R}^d\rightarrow\mathbb{R}^d$ is a global diffeomorphism. Thus, the zero $\x^*$ of $\nabla f$ exists and is unique. Moreover, since $\nabla^2 f(\x^*)$ has $k$ negative eigenvalues, $\x^*$ is an index-k saddle point.

The proof of the convergence is similar to that of \Cref{descent with exact eigenvector}, but the inequality in \eqref{descent part exact} should be changed into
$$
\begin{aligned}   &\langle\mathcal{P}_{V}\nabla f(\x(n)),\x(n)-\x^*\rangle\\
&=\langle\mathcal{P}_{V} \nabla^2 f(\x)(\x(n)-\x^*),\x(n)-\x^*\rangle\\
&\quad \quad +\left\langle\mathcal{P}_{V}\int_0^1 (\nabla^2 f(\x^*+s(\x(n)-\x^*))-\nabla^2 f(\x)) \mathrm{d} s(\x(n)-\x^*),\x(n)-\x^*\right\rangle\\
    &\geqslant \left(\mu-M_1\right)\Vert \x(n)-\x^* \Vert^2_2.
    \end{aligned}
$$
Assume $\mathcal{F}(n)$ is a filtration, we have 
    $$
    \begin{aligned}
    &\mathbb{E}(\Vert\x(n+1)-\x^*\Vert_2^2|\mathcal{F}(n))\\
        &\leqslant\Vert \x(n)-\x^* \Vert^2_2-2\alpha(n)\langle \mathcal{P}_{V}\nabla f(\x(n)),\x(n)-\x^*  \rangle+2\alpha(n)^2\sigma^2+2\alpha(n)^2\Vert\nabla f(\x)\Vert_2^2\\
        &\leqslant\Vert \x(n)-\x^* \Vert^2_2-2\alpha(n) \left(\mathcal{P}_{V} \nabla^2 f(\x)(\x(n)-\x^*),\x(n)-\x^*\rangle-M_1\Vert \x(n)-\x^* \Vert^2_2\right)\\
        &\quad+2\alpha(n)^2\sigma^2+2\alpha(n)^2\Vert\nabla f(\x)\Vert^2_2\\
        &\leqslant (1-2\alpha(n)(\mu-M_1))\Vert\x(n)-\x^*\Vert^2+2\alpha(n)^2\sigma^2+2L^2\alpha(n)^2\Vert\x(n)-\x^*\Vert_2^2\\
        &=(1-2\alpha(n)(\mu-M_1)+2L^2\alpha(n)^2)\Vert\x(n)-\x^*\Vert^2_2+2\alpha(n)^2\sigma^2\\
        &\leqslant (1-\alpha(n)(\mu-M_1))\Vert\x(n)-\x^*\Vert^2_2+2\alpha(n)^2\sigma^2.
    \end{aligned}
    $$ 
The last inequality follows from the assumption that, without loss of generality, $\alpha(n)$ is sufficiently small such that $2L^2\alpha(n)^2\leqslant \alpha(n)(\mu-M_1)$.
By applying \Cref{lemma: An and Bn}, $\|\x(n) - \x^*\|_2^2$ converges to zero almost surely.
\end{proof}
\begin{remark}
    A direct consequence of \Cref{Global convergence with exact eigenvector} is that if the objective function is a quadratic function or close to a quadratic function, then we can obtain global almost sure convergence of the stochastic saddle-search algorithm.
\end{remark}

\subsection{Inexact eigenvector} 
\label{sec: Inexact eigenvector}
In actual implementation, we can not expect to have an exact eigenvector due to the round-off  error and the tolerance set in the iteration methods for eigenvectors (such as
 the classical power iteration, randomized SVD \cite{halko2011finding}, and locally optimal block preconditioned conjugate gradient method \cite{Knyazev2001TowardTO}). 
Moreover, in some cases, the exact Hessian may not be accessible. Or, even if it is available, the high-dimensional settings can be computationally expensive. Thus, a stochastic, unbiased low-rank approximation of the Hessian can reduce the computational cost.

In the following, we first introduce, see \Cref{Stochastic eigenvector-search}, a stochastic eigenvector-search algorithm designed to find approximate unstable directions that are sufficiently close to the true unstable eigenvectors. Then, in \Cref{Convergence of the stochastic saddle-search algorithm with inexact unstable directions}, we prove that, with high probability, the saddle-search algorithm exhibits local convergence based on these approximate unstable directions.

\subsubsection{Stochastic eigenvector-search algorithm}\label{Stochastic eigenvector-search}
When only a stochastic unbiased estimate of a matrix is available, how can we recover its eigenvectors? 
One possible strategy is to adopt stochastic algorithms for principal component analysis (PCA). The use of stochastic approximation for PCA dates back to the work in \cite{oja1985stochastic}, followed by a series of developments in convergence analysis and algorithm design \cite{li2018near,shamir2015stochastic,arora2012stochastic,arora2013stochastic}. Most convergence analyses of the stochastic PCA methods rely on the assumption of a simple eigenvalue to ensure the uniqueness of the leading eigenvector. However, in the context of saddle dynamics, our focus shifts from identifying a single eigenvector to recovering the unstable eigenspace. It is important to note that the projection matrix $\sum_{i=1}^k \vvec_i \vvec^\top_i$ is invariant under rotations within the eigenspace corresponding to repeated eigenvalues. 

In this subsection, we focus on a fixed Hessian matrix $\nabla^2f$, independent of the iterate $\x$, and its stochastic counterpart. To emphasize this independence, we adopt the more general notation $\Hvec$ in place of $\nabla^2f$. The dynamics $\dot{\vvec}=-(\I-\vvec\vvec^\top)\Hvec\vvec$ is the gradient flow of the Rayleigh quotient, 
$$
	\min_{{\vvec}}\text{  }\left<\vvec, \Hvec\vvec\right>,\ \text{s.t.}\ \Vert \vvec \Vert_2^2=1,
$$
which can be used to find the eigenvector of $\Hvec$ corresponding to the smallest eigenvalue. We consider the stochastic version:
\begin{equation}
    \hat{\vvec}(n+1)=\vvec(n)-\alpha(n)(\I-\vvec(n)\vvec(n)^\top)\Hvec(\omega(n))\vvec(n),
    \vvec(n+1)=\frac{\hat{\vvec}(n+1)}{\Vert \hat{\vvec}(n+1)\Vert_2},
    \label{stochastic eigenvector}
\end{equation}
where $\Hvec(\omega)$ is a stochastic matrix to approximate $\Hvec$, which satisfies the following unbiased assumption.

\begin{assumption}\label{step size assumption for v}
$\mathbb{E}(\Hvec(\omega))=\Hvec,\mathbb{E}(\Vert\Hvec(\omega)\Vert_2^2)\leqslant \sigma^2$, $\Vert\Hvec(\omega)\Vert_2\leqslant G_1$.
\end{assumption}

The convergence of \eqref{stochastic eigenvector} under the step size assumption in \Cref{step size sec 1} and \Cref{step size assumption for v} has been studied in \cite{oja1985stochastic} under the assumption of a simple eigenvalue. The following proposition shows that convergence still holds for repeated negative eigenvalues, with the convergence of the sequence to a single eigenvector relaxed to its limiting points lying within the eigenspace.

\begin{proposition}\label{a.s. convergence of the eigenvector searching}
    The limit points of $\mathbf{v}(n)$ generated by \eqref{stochastic eigenvector}, under \Cref{step size sec 1} and \Cref{step size assumption for v}, almost surely lie in the eigenspace of $\Hvec$. 
\end{proposition}
\begin{proof}
The proof is given in the Supplementary Materials. The main step is to define the Rayleigh quotient as $A(\vvec)=\vvec^\top \Hvec \vvec$ and set $B(\vvec)=\|(\I-\vvec \vvec^\top)\Hvec \vvec\|_2^2$. One can then verify that $A(\vvec(n)), B(\vvec(n))$ satisfy the conditions of \Cref{lemma: An and Bn}, from which the result follows.
\end{proof}

To identify the $k$ unstable directions in the index-$k$ saddle-search algorithm, one may proceed iteratively: begin by finding $\vvec_1$ as in \Cref{a.s. convergence of the eigenvector searching}, then compute $\vvec_2$ utilizing the previously obtained $\vvec_1$, and continue in this manner. Actually, $\{\vvec_{\bar{k}}\}_{\bar{k}\geqslant 2}^k$ can be obtained iteratively from
$$
\begin{aligned}
    \min_\vvec\ & \vvec^\top\Hvec \vvec \\
    \text{s.t. } \Vert \vvec  \Vert^2=1, \vvec_j^T&\vvec=0, \;j=1,\cdots, \bar{k}-1.
\end{aligned}
$$
The corresponding Lagrangian function and the gradient flow are given by
$$
    \mathcal{L}(\vvec,\{\mu_j\}^{\bar{k}}_1)=
    \vvec^\top\Hvec \vvec-\mu_1(\Vert \vvec  \Vert^2_2-1)-\sum_{j=2}^{\bar{k}} \mu_j\vvec_{j-1}^T\vvec,
$$
 $$   \dot{\vvec}_{\bar{k}}=-\frac{1}{2}\frac{\partial}{\partial \vvec_{\bar{k}}}
 \mathcal{L}(\vvec_{\bar{k}},\{\mu_j\}_1^{\bar{k}})=
 \Hvec \vvec_{\bar{k}}-\mu_1 \vvec_{\bar{k}}-\frac{1}{2}\sum_{j=2}^{\bar{k}}\mu_j\vvec_{j-1},
$$
where $\mu_1=\vvec_{\bar{k}}^\top \Hvec\vvec_l$ (derived from $\frac{d\Vert\vvec_{\bar{k}}\Vert^2_2}{dt}=0$), $\mu_j=2\vvec_{j-1}^\top \Hvec\vvec_{\bar{k}}$ (derived from $\frac{d\langle \vvec_{j-1},\vvec_{\bar{k}} \rangle}{dt}=0$, for $2\leqslant j\leqslant {\bar{k}}$). Thus, the stochastic eigenvector search iteration for $\vvec_{\bar{k}}$,$\bar{k}=2,\cdots,k$ is defined as
\begin{equation}
\begin{aligned}
    \hat{\vvec}(n+1)=\vvec(n)-\alpha(n)&\left(\I-\vvec(n)\vvec(n)^\top-\sum_{i=1}^{\bar{k}-1}\vvec_i\vvec_i^\top\right)\Hvec(\omega(n))\vvec(n),\\
    &\vvec(n+1)=\frac{\hat{\vvec}(n+1)}{\Vert \hat{\vvec}(n+1)\Vert_2}.
    \label{stochastic eigenvector larger}
    \end{aligned}
\end{equation}

\begin{proposition}
    The limit points, $\vvec_\infty \notin span\{\vvec_1,\cdots,\vvec_{\bar{k}-1}\}$, of the iteration points $\vvec(n)$ generated by \eqref{stochastic eigenvector larger} with \Cref{step size sec 1} and \Cref{step size assumption for v}, almost surely lie in the eigenspace of $\Hvec$, provided that $\vvec_i,i=1,\cdots,\bar{k}-1$ are eigenvectors of $\Hvec$ and $\vvec(0)$ is orthogonal to $\vvec_i,i=1,\cdots,\bar{k}-1$.
\end{proposition}
\begin{proof}
    The proof is similar to that of \Cref{a.s. convergence of the eigenvector searching}, the iterations have a subsequence that almost surely converges to a unit vector $\vvec_\infty$, which makes $\Vert (\I-\vvec_\infty\vvec_\infty^\top-\sum_{i=1}^{\bar{k}-1}\vvec_i\vvec_i^\top)\Hvec\vvec_\infty  \Vert_2$ vanish. $\vvec(n)$ is always orthogonal to $\vvec_i,i=1,\cdots,\bar{k}-1$, because if $\vvec(n-1)$ is orthogonal to $\vvec_i,i=1,\cdots,\bar{k}-1$, then
    $$\left\langle\left(\I-\vvec(n)\vvec(n)^\top-\sum_{i=1}^{\bar{k}-1}\vvec_i\vvec_i^\top\right)\Hvec(\omega(n))\vvec(n),\vvec_j\right\rangle=0,j=1,\cdots,\bar{k}-1$$ and $\vvec(n)$ is also orthogonal to $\vvec_i,i=1,\cdots,\bar{k}-1$, which means $\vvec_\infty$ is orthogonal to $\vvec_i,i=1,\cdots,\bar{k}-1$. If $\vvec_i,i=1,\cdots,\bar{k}-1$ are eigenvectors of $\Hvec$, then $\vvec_i^\top \Hvec \vvec_\infty=\lambda_i\vvec_i^\top \vvec_\infty=0$, and
    $$
        \Hvec \vvec_\infty-(\vvec_\infty^T\Hvec\vvec_\infty)\vvec_\infty=\sum_{i=1}^{\bar{k}-1}(\vvec_i^\top \Hvec \vvec_\infty)\vvec_i=0,
    $$
which indicates that $\vvec_\infty$ is an eigenvector of $\Hvec$. 
\end{proof}

Now, we can give the stochastic eigenvector-search algorithm to find $k$ unstable directions in \Cref{algorithm vector}.
\begin{algorithm}
\caption{Stochastic eigenvector-search algorithm for $k$ unstable directions}
\label{algorithm vector}
\begin{algorithmic}
\STATE{Initial condition: $\vvec_i,i=1,\cdots,k$, $n_\vvec\leftarrow0$}
\STATE{Tolerance: $\epsilon_\vvec>0$}

\WHILE{$\Vert(\I-\vvec_1\vvec_1^\top)\Hvec\vvec_1\Vert^2_2\geqslant L^2\epsilon_\vvec$}
\STATE{$\vvec_1\leftarrow\vvec_1-\alpha(n_\vvec)(\I-\vvec_1\vvec_1^\top)\Hvec(\omega(n_\vvec))\vvec(n_\vvec)$}
\STATE{$\vvec_1\leftarrow\vvec_1/\Vert \vvec_1\Vert_2$}
\STATE{$n_\vvec\leftarrow n_\vvec+1$}
\ENDWHILE

\FOR{$j$ in $\{2,\cdots,k\}$}
\STATE{$n_\vvec\leftarrow0$,$\vvec_j\leftarrow\vvec_j-(\I-\sum_{i=1}^{j-1}\vvec_i\vvec_i^\top)\vvec_j$, $\vvec_j\leftarrow\vvec_j/\Vert \vvec_j\Vert_2$}

\WHILE{$\Vert(\I-\vvec_j\vvec_j^\top-\sum_{i=1}^{j-1}\vvec_i\vvec_i^\top)\Hvec\vvec_j\Vert^2_2\geqslant L^2\epsilon_\vvec$}
\STATE{$\vvec_j\leftarrow\vvec_j-\alpha(n_\vvec)(\I-\vvec_j\vvec_j^\top-\sum_{i=1}^{j-1}\vvec_i\vvec_i^\top)\Hvec(\omega(n_\vvec))\vvec_j$}
\STATE{$\vvec_j\leftarrow\vvec_j/\Vert \vvec_j\Vert_2$}
\STATE{$n_\vvec\leftarrow n_\vvec+1$}
\ENDWHILE
\ENDFOR
\end{algorithmic}
\end{algorithm}
\begin{remark}
The termination tolerance is set to $L^2 \epsilon_\vvec$ ($L$ is defined in \Cref{assumption: Hessian}) to eliminate the influence of multiplicative scaling of the objective function. For example, $Cf(\x)$ and $f(\x)$ should be regarded as equivalent, even though their squared norm of Hessians differ by a factor of $C^2$. The parameter $\epsilon_\vvec$ serves as an accuracy threshold which is independent of the scaling of the objective function.
\end{remark}

\begin{remark}
    Consider the Rayleigh quotient problem 
$$
\begin{aligned}
    \min_\vvec \ & \vvec^\top\Hvec \vvec \\
    \text{s.t. } \Vert \vvec  \Vert^2_2=1, \vvec_i^T\vvec&=0, i=1,\cdots,\bar{k}-1,
\end{aligned}
$$
to find the eigenvector $\vvec_{\bar{k}}$ corresponding to the $\bar{k}$-th smallest eigenvalue, where $\vvec_i,i=1,\cdots,\bar{k}-1$ is the eigenvector corresponding to the $i$-th smallest eigenvalue. If the $k$-th smallest eigenvalue $\lambda_k$ is simple,  $\vvec_{\bar{k}}$ serves as the global minimizer, and $\vvec_i, i>\bar{k}$ serve as the saddle point. The stochastic gradient descent (SGD) has an avoidance feature, i.e., it avoids converging to a saddle point \cite{daneshmand2018escaping}. 
The numerical
results reported in \Cref{sec: numerical results} also support this point. Even in the worst situation, i.e. the eigenvector-search algorithm converges to a stable eigen-direction, $\vvec_{i_2},i_2>\bar{k}$, of $\Hvec=\nabla^2 f(\x)$ at $\x \in U$ (defined in \Cref{assumption: Hessian}), we can easily identify that it is not an
unstable direction from $\vvec_{i_2}^\top\Hvec\vvec_{i_2}>0$. We may then repeat the stochastic eigenvector search with 
$$
\begin{aligned}
    \tilde{\vvec}(n+1)=\vvec(n)-&\alpha(n)\left(\I-\vvec(n)\vvec(n)^\top-\sum_{i=1}^{\bar{k}-1}\vvec_i\vvec_i^\top-\vvec_{i_2}\vvec_{i_2}^\top\right)\Hvec(\omega(n))\vvec(n),\\
&\vvec(n+1)=\frac{\tilde{\vvec}(n+1)}{\Vert \tilde{\vvec}(n+1)\Vert},
\end{aligned}
$$
until we find $k$ unstable directions with a restart at most $d$ times. This also explains why we identify unstable directions sequentially, rather than attempting to minimize $k$ Rayleigh quotients 
simultaneously to obtain $k$ unstable directions in parallel as in \eqref{eq: SD} and \cite{oja1985stochastic,oja1992principal}. Thus, in the following, we assume that the $k$ approximate unstable directions in \Cref{algorithm vector} are actually close to the $k$ smallest eigenvectors of $\Hvec$.
\end{remark}

There is a necessary tolerance $L^2\epsilon_\vvec$ in \Cref{algorithm vector}, and we need to estimate how this tolerance affects the precision of the unstable directions.
\begin{proposition}\label{error of the inexact eigenvector}
    If a set of orthonormal vectors
 $\{ \vvec_i\}_{i=1}^{\bar{k}-1} \in \mathbb{S}^{d-1}$ has a small enough error $\theta$ with a set of exact eigenvectors $\{\vvec_i^*\}_{i=1}^{\bar{k}-1}\in \mathbb{S}^{d-1}$ of $\Hvec$,
     i.e. 
    $$
    1-(\vvec_i^\top\vvec_i^*)^2<\theta,
    \text{ or } \vvec_i=\sum_{j=1}^d \alpha_{ij} \vvec^*_j, \alpha_{ii}^2>1-\theta,\sum_{j=1,j\neq i}^d\alpha_{ij}^2<\theta, i=1,\cdots,\bar{k}-1,
    $$
    with $\vvec^*_j$ being the eigenvector of $\Hvec$ corresponding to the $j$-th smallest eigenvalue $\lambda_j$, and the eigenvalues satisfy \eqref{eq:eigenvalues}.
    Then, under \Cref{step size sec 1} and \Cref{step size assumption for v}, the limit points, $\vvec_\infty$, of the iteration points $\vvec(n)$ generated by \eqref{stochastic eigenvector larger} 
is almost surely close to the eigenspace of $\lambda_{i^*}$, where $i^*=\min(\mathrm{argmin}_{1\leqslant i\leqslant d}(\lambda_i-\vvec_\infty^T\Hvec\vvec_\infty)^2).$ Specifically, define $V_{i^*}=[\vvec^*_{i^*},\cdots,\vvec^*_{i^*+n(\lambda^*)-1}]$, where each column is an eigenvector corresponding to the eigenvalue $\lambda_{i^*}$, and let $n(\lambda^*)$ be the algebraic multiplicity of the eigenvalue $\lambda^*$, we have, for $z_{\bar{k}}:=\sqrt{\bar{k}^2L^2\theta/d}$,
    \begin{equation}
    1-\Vert \vvec_{\infty}^\top V_{i^*}\Vert^2_2\leqslant \frac{z_{\bar{k}}^2 d}{\min\left(\left(\Delta-z_{\bar{k}}\right)^2,\left(2\mu-z_{\bar{k}}\right)^2\right)},
    \label{error between converged vector and eigenvector}
    \end{equation}
where $\Delta=\min_{\lambda_i<\lambda_j<0} \lambda_j-\lambda_i$. If the algorithm \Cref{algorithm vector} stops with $$ \Vert
\mathcal{P}_\vvec
\Hvec\vvec(n_{end}) \Vert^2_2< L^2\theta, \text{ where } \mathcal{P}_\vvec=\I-\vvec(n_{end})\vvec(n_{end})^\top-\sum_{i=1}^{\bar{k}-1}\vvec_i\vvec_i^\top$$
at the $n_{end}$-th step, then \eqref{error between converged vector and eigenvector} also holds when replacing $\vvec_\infty$ by $\vvec(n_{end})$.
\end{proposition}
\begin{proof} 
The proof is given in the Supplementary Materials.
\end{proof}

\begin{theorem}\label{distance between projection}
Under \Cref{assumption: Hessian} and \Cref{step size assumption for v}, the \Cref{algorithm vector} with $\Hvec=\nabla^2f(\x)$ for some $\x \in U$ almost surely stops in finite steps, and the output $\{\tilde{\vvec}_j\}_{1}^k$ are approximations of the exact unstable eigenvectors $\{\vvec_j\}_{1}^k$ with  an error bounded by the unscaled tolerance $\epsilon_\vvec$. Specifically,  
let $\Delta=\min_{\lambda_i<\lambda_j<0} \lambda_j-\lambda_i$ be the minimal gap of distinct negative eigenvalues of $\Hvec$ and $Q=\max\left(1,\frac{1}{\min(\Delta^2/4,\mu^2)},\frac{1}{L^2}\right),$ 
if $\epsilon_\vvec$ is small enough to satisfy     $\bar{z}_k:=\sqrt{L^{2k}Q^k(k!)^2\epsilon_\vvec/d}<\min(\Delta/2,\mu),$
so that
\begin{equation}\label{smalll tolerance}
{\min\left(\left(\Delta-\bar{z}_k\right)^2,\left(2\mu-\bar{z}_k\right)^2\right)}>1/Q,
\end{equation}
then
$$
    \left\Vert \sum_{i=1}^k \tilde{\vvec}_i \tilde{\vvec}_i^\top -\sum_{i=1}^k \vvec_i \vvec_i^\top \right\Vert_2^2\leqslant k \bar{z}_k^2 d.
$$
\end{theorem}
\begin{proof}

We first prove that $
1- (\tilde{\vvec}_m^\top\vvec_m)^2\leqslant \bar{z}_m^2 d,\; m\leqslant k$ by induction. For $m=1$, $\tilde{\vvec}_1$ satisfies the tolerance in \Cref{algorithm vector}, i.e.,
$\tilde{\vvec}_1$ satisfies 
$$
    \Vert\Hvec \tilde{\vvec}_1-(\tilde{\vvec}_1^T\Hvec\tilde{\vvec}_1)\tilde{\vvec}_1 \Vert_2^2\leqslant L^2\epsilon_\vvec.
$$
By the same steps in the proof of \Cref{error of the inexact eigenvector} (provided in the Supplementary Materials), we can get
$1- (\tilde{\vvec}_1^\top\vvec_1)^2\leqslant L^2 Q\epsilon_\vvec =\bar{z}_1^2 d.$
Assume for all $1\leqslant m\leqslant\bar{m}\leqslant k-1$, $1- (\tilde{\vvec}_m^\top\vvec_m)^2\leqslant \bar{z}_m^2 d$ holds. Note that $L^2Q\geqslant 1$, so $\bar{z}_m \leqslant \bar{z}_{\bar{m}}$ and  
$1- (\tilde{\vvec}_m^\top\vvec_m)^2\leqslant \bar{z}_m^2 d\leqslant \bar{z}_{\bar{m}}^2 d$.
Then for $\bar{m}+1$, 
we can set $\theta= \bar{z}_{\bar{m}}^2 d
$ in \Cref{error of the inexact eigenvector}, 
and
$$
\begin{aligned}
1-(\tilde{\vvec}_{\bar{m}+1}^\top\vvec_{\bar{m}+1})^2&\leqslant \frac{
(\bar{m}+1)^2L^2\theta
}{\min\left(\left(\Delta-\bar{z}_{\bar{m}+1}
\right)^2,\left(2\mu-\bar{z}_{\bar{m}+1}\right)^2\right)}\\
& \leqslant Q (\bar{m}+1)^2L^2\theta
=  \bar{z}_{\bar{m}+1}^2 d.
\end{aligned}
$$
Let $V_\perp$ denote the matrix whose columns form the orthogonal complement of $V$,
$$
    \left \Vert \sum_{i=1}^k \tilde{\vvec}_i \tilde{\vvec}_i^\top -\sum_{i=1}^k \vvec_i \vvec_i^\top \right\Vert_2^2=\Vert\tilde{V}^TV_\perp\Vert_2^2=\max_{\Vert\x\Vert^2=1}\Vert\tilde{V}^TV_\perp \x\Vert^2_2\leqslant k 
    \bar{z}_k^2 d.
$$
The last inequality follows from noting that the $i$-th component of $\tilde{V}^TV_\perp \x$ satisfies $[\tilde{V}^TV_\perp \x]_i^2\leqslant1-(\tilde{\vvec}_{i}^\top\vvec_{i})^2\leqslant \bar{z}_k^2 d.
$
\end{proof}

\subsubsection{Convergence of the stochastic saddle-search algorithm with inexact unstable directions}\label{Convergence of the stochastic saddle-search algorithm with inexact unstable directions}

With the stochastic eigenvector-search algorithm, we propose the stochastic saddle-search algorithm in \Cref{algorithm}. Stochastic eigenvector-search algorithms in \Cref{algorithm vector} can be used to identify approximate (but inexact) unstable directions. A natural question is whether the stochastic saddle-search algorithm in \Cref{algorithm} can locate the target saddle point with these approximate unstable directions. We consider this question in the remainder of this section.
\begin{algorithm}
\caption{Index-$k$ stochastic saddle-search algorithm with stochastic eigenvector search}
\label{algorithm}
\begin{algorithmic}
\STATE{Initial condition:  $\x,\vvec_i,i=1,\cdots,k$, $n_\x\leftarrow 0$}
\STATE{Tolerance: $\epsilon_\x>0,\epsilon_\vvec>0$}
\WHILE{$\Vert\nabla f(\x)\Vert_2^2\geqslant L^2 \epsilon_\x$}
\STATE{$\x\leftarrow \x-\alpha(n_\x)(\I-\sum_{i=1}^k2\vvec_i\vvec_i^T)\nabla f(\x,\omega(n_\x))$}
\STATE{$\{\vvec_i\}_1^k \leftarrow$ searching $k$ unstable directions of $\Hvec=\nabla^2f(\x)$ with \Cref{algorithm vector}}

\STATE{$n_\x\leftarrow n_\x+1$}
\ENDWHILE
\end{algorithmic}
\end{algorithm}

From \Cref{distance between projection}, the distance between the unstable directions calculated from stochastic eigenvector searching ($\tilde{\vvec}_i$) and the exact unstable directions ($\vvec_i$) can be bounded by $\epsilon_\vvec$. For brevity in notation, we assume that there is a small $1\gg\theta>0$ which satisfies 
$$
    \Vert\tilde{V}-V\Vert_2^2=\left\Vert \sum_{i=1}^k \tilde{\vvec}_i \tilde{\vvec}_i^\top -\sum_{i=1}^k \vvec_i \vvec_i^\top \right\Vert_2^2\leqslant \theta=k \bar{z}_k^2 d=O(\epsilon_\vvec).
$$

\begin{proposition}[Global convergence with slightly changed Hessian and inexact eigenvector]\label{Global convergence with inexact eigenvector}
    Assume that the eigenvalues of $\nabla^2 f(\x)$ satisfy \eqref{eq:eigenvalues}
and $\Vert\nabla^2 f(\x)-\nabla^2 f(\x^*)\Vert_2\leqslant M_3<(1-\sqrt{\theta})\mu-L(\theta+5\sqrt{\theta}),\forall \x \in \mathbb{R}^d$, then $f$ has a unique critical point $\x^*$ which is an index-k saddle point. With any initial condition, the stochastic saddle-search algorithm with \Cref{step size sec 1} and \Cref{unbias} almost surely converges to $\x^*$.
\end{proposition}

\begin{proof}
The proof is similar to \Cref{Global convergence with exact eigenvector}. We have
    $$
    \begin{aligned}
        &\mathbb{E}(\Vert\x(n+1)-\x^*\Vert_2^2|\mathcal{F}(n))\\
        &\leqslant\Vert \x(n)-\x^* \Vert^2_2-2\alpha(n)\langle \mathcal{P}_{\tilde{V}}\nabla f(\x(n)),\x(n)-\x^*  \rangle+2\alpha(n)^2\sigma^2+2\alpha(n)^2\Vert\nabla f(\x)\Vert_2^2\\
        &\leqslant(1-2\alpha(n)(\Vert  \mathcal{P}_{\tilde{V}}\nabla^2 f(\x(n))\Vert_2-M_3)+2\alpha(n)^2L^2)\Vert \x(n)-\x^* \Vert^2_2+2\alpha(n)^2\sigma^2.
    \end{aligned}
$$ 
Because $\tilde{V}=\sum_{i=1}^k \tilde{\vvec}_i \tilde{\vvec}_i^\top$ is an approximation of $V=\sum_{i=1}^k \vvec_i \vvec_i^\top$, $\mathcal{P}_{\tilde{V}}\nabla^2 f(\x(n))$ is close to the positive definite matrix $\mathcal{P}_{V} \nabla^2f(\x(n))=\sum_{i=1}^k-\lambda_i\vvec_i\vvec_i^\top+\sum_{i=k+1}^d\lambda_i\vvec_i\vvec_i^\top$ with the approximation error controlled by the parameter $\theta$. Specifically, from \cite[Theorem 5.4]{luo2022sinum}, 
\begin{equation}
\begin{aligned}     
&\mathcal{P}_{\tilde{V}}\nabla^2f(\x(n))=(\sqrt{\theta}-1)\sum_{i=1}^k\lambda_i\vvec_i\vvec_i^\top+(1-\theta)\sum_{i=k+1}^d \lambda_i\vvec_i\vvec_i^\top\\
    &\qquad +    \underbrace{\left(\mathcal{P}_{\tilde{V}} \nabla^2f(\x(n))-(\sqrt{\theta}-1)\sum_{i=1}^k\lambda_i\vvec_i\vvec_i^\top-(1-\theta)\sum_{i=k+1}^d \lambda_i\vvec_i\vvec_i^\top\right)}_{\Vert \cdot \Vert_2\leqslant L(\theta+5\sqrt{\theta})}.
    \end{aligned}    \label{eq: pertubation of inexact eigendirection}
\end{equation}
Then, assume that $2\alpha(n)^2L^2\leqslant \alpha(n)((1-\sqrt{\theta})\mu-L(\theta+5\sqrt{\theta})-M_3)$ without loss of generality, we have 
    \begin{equation*}
    \begin{aligned}
        &\mathbb{E}(\Vert\x(n+1)-\x^*\Vert_2^2|\mathcal{F}(n))\\
        &\leqslant(1-2\alpha(n)(\Vert  \mathcal{P}_{\tilde{V}}(\nabla^2 f(\x(n))\Vert_2-M_3)+2\alpha(n)^2L^2)\Vert \x(n)-\x^* \Vert^2_2+2\alpha(n)^2\sigma^2.\\
        &\leqslant(1-\alpha(n)((1-\sqrt{\theta})\mu-L(\theta+5\sqrt{\theta})-M_3))\Vert \x(n)-\x^* \Vert^2_2+2\alpha(n)^2\sigma^2.
    \end{aligned}
    \end{equation*} 
    By applying \Cref{lemma: An and Bn}, $\|\x(n) - \x^*\|_2^2$ converges to zero almost surely.
\end{proof}

Given a general objective function with only local properties known near the target saddle point as in \Cref{assumption: Hessian}, we present the following proposition, which is a key to establish the convergence of the iterates to the target saddle point.

\begin{proposition}\label{descent with inexact eigenvector}
    Suppose the \Cref{assumption: Hessian} holds. Let $\theta>0$ be a small enough parameter such that $(1-\sqrt{\theta})\mu-L\theta-5L\sqrt{\theta}>0$ and $\alpha(n)<\frac{1}{2\mu}$. Define a neighborhood of $\x^*$,  
    \begin{equation}\label{attraction domain with inexact eigenvector}
        U=\left\{\x,\Vert \x-\x^* \Vert_2 \leqslant \min\left( \frac{(1-\sqrt{\theta})\mu-L\theta-5L\sqrt{\theta}}{M},\delta\right)\right\}.
    \end{equation}
    If $\x(n)\in U$, we have
\begin{equation}\label{descent term}
    \begin{aligned}
        \Vert\x(n+1)-\x^*\Vert^2_2/2\leqslant& (1-\alpha(n)[(1-\sqrt{\theta})\mu-L\theta-5L\sqrt{\theta}])\Vert\x(n)-\x^*\Vert^2_2/2\\
&+\alpha(n)\psi(n)+\frac{1}{2}\alpha(n)^2\Vert\nabla f(\x(n);\omega(n))\Vert^2_2,
        \end{aligned}
\end{equation}
where $\psi(n) =-\langle \mathcal{P}_{\tilde{V}}(\nabla f(\x(n);\omega(n))-\nabla f (\x(n))), \x(n) -\x^* \rangle$ is a martingale sequence with zero expectation.
\end{proposition}

\begin{proof}
The proof follows the same steps as in \Cref{descent with exact eigenvector}, except that \eqref{descent part exact} is replaced by the following inequality. By utilizing \eqref{eq: pertubation of inexact eigendirection}, we have 
$$
\begin{aligned}
&\langle\mathcal{P}_{V}\nabla f(\x(n)),\x(n)-\x^*\rangle\\
    &\geqslant \left((1-\sqrt{\theta})\mu-L(\theta+5\sqrt{\theta})-\frac{M\Vert \x(n)-\x^* \Vert_2}{2}\right)\Vert \x(n)-\x^* \Vert^2_2\\
    &\geqslant\frac{(1-\sqrt{\theta})\mu-L(\theta+5\sqrt{\theta})}{2}\Vert \x(n)-\x^* \Vert^2_2.
    \end{aligned}
$$
\end{proof}

\begin{corollary}\label{local a.s. convergence with inexact eigenvector}
Same as \Cref{local convergence with exact eigenvector}, if the \Cref{step size sec 1} and \Cref{unbias} hold, with \Cref{descent with inexact eigenvector}, we have the almost sure convergence of the iterates to the saddle point conditioned on the boundedness event $E^\infty = \{ \x(n) \in U, \forall n \}$. 
\end{corollary}

A natural question is whether the event $E^\infty$ occurs with a positive probability. To this end, we establish the following proposition about the attraction of the saddle point, which is analogous to discussions in \cite[Theorem 4]{mertikopoulos2020almost} for the SGD method. 

\begin{proposition}\label{attraction domain}
     Suppose the \Cref{step size sec 1}, \Cref{unbias}, and \Cref{assumption: Hessian} hold. Assume $\theta>0$ is small enough such that $(1-\sqrt{\theta})\mu-L\theta-5L\sqrt{\theta}>0$. If we further assume that the stochastic saddle-search algorithm is implemented with the step size satisfying that $\sum_{n=0}^\infty \alpha(n)^2$ is sufficiently small (e.g. $\alpha(n)=\frac{\gamma}{(n+m)^p}$ with a large enough $m$ and $p\in(1/2,1]$). Set $\epsilon$ as
$$ 4\epsilon+2\sqrt{\epsilon}=\min\left([(1-\sqrt{\theta})\mu-L\theta-5L\sqrt{\theta}]/M,\delta\right)^2, 
$$
and define \begin{equation}\label{definition of U0}
   U_0=\left\{\x,\Vert\x-\x^*\Vert_2\leqslant\sqrt{2\epsilon}\right\}\subset U,
   \end{equation}
where $U$ is defined in \eqref{attraction domain with inexact eigenvector}. Then, for $\x(n)$ generated by the stochastic saddle-search algorithm with an initial condition $\x(0) \in U_0$, the event  
    $$
        E^\infty = \{ \x(n) \in U, \forall n \}
   $$
    occurs with a positive probability. In a particular case, if the stochastic saddle-search algorithm is run with a step-size schedule of the form $\alpha(n)=\frac{\gamma}{(n+m)^p},p\in(1/2,1]$ and large enough $m$, then $\mathbb{P}(E^\infty| \x(0) \in U_0) >1-\epsilon_{E^\infty}$, and $1-\epsilon_{E^\infty}\rightarrow1$ as $m\rightarrow \infty$.  
\end{proposition}

\begin{proof}
   Since \Cref{descent with inexact eigenvector} is analogous to \cite[Proposition D.1]{mertikopoulos2020almost}, which constitutes a key step in establishing the positive probability of the event $E^\infty$, the proof proceeds along the same lines as \cite[Theorem 4]{mertikopoulos2020almost}. We give the proof in the Supplementary Materials. 
\end{proof}

\begin{proposition}[Convergence rate]\label{convergence rate}
    Suppose the \Cref{unbias} and \Cref{assumption: Hessian} hold. If the stochastic saddle-search algorithm is run with $\alpha(n)=\frac{\gamma}{n+m}$, where $m,\gamma>0$ are large enough. Conditioned on $E^\infty$, the sequence $\x(n)$ generated by \Cref{algorithm} satisfies:
    $$\mathbb{E}(\Vert \x(n)-\x^* \Vert^2_2|E^\infty)\leqslant \frac{\bar{M}}{(n+m)\mathbb{P}(E^\infty)},$$
    where $\bar{M}=\frac{2\gamma^2(G^2+\sigma^2)}{[(1-\sqrt{\theta})\mu-L\theta-5L\sqrt{\theta}]\gamma-1}\geqslant \frac{8(G^2+\sigma^2)}{[(1-\sqrt{\theta})\mu-L\theta-5L\sqrt{\theta}]^2},$ and the equality holds when $\gamma=\frac{2}{(1-\sqrt{\theta})\mu-L\theta-5L\sqrt{\theta}}.$
\end{proposition}

\begin{proof}
The proof is based on \Cref{descent with inexact eigenvector} and \Cref{attraction domain}, and is provided in the Supplementary Materials.
\end{proof}

\begin{remark}
It is worth noting that the convergence rate is closely tied to the constant $\bar{M}\geqslant \frac{8(G^2+\sigma^2)}{[(1-\sqrt{\theta})\mu-L\theta-5L\sqrt{\theta}]^2}$, which intuitively demonstrates that large stochastic gradient noise (reflected in a large 
$\sigma^2$) and large errors in estimating the unstable subspace (reflected in a small $[(1-\sqrt{\theta})\mu-L\theta-5L\sqrt{\theta}]^2$) slow down the convergence.
\end{remark}

From the above discussion, similar to SGD \cite{bertsekas2000gradient,chung1954stochastic, mertikopoulos2020almost,polyak1992acceleration}, the stochastic saddle-search algorithm, as a special case of stochastic gradient-type methods, achieves convergence through a noisy recursive process obtained in \Cref{descent with inexact eigenvector}. However, different from SGD, the saddle-search algorithm requires the estimation of the unstable subspace, and the accuracy of this estimation critically influences both the convergence radius ($U_0$ defined in \eqref{definition of U0}) and the convergence rate. 

Combining all of the above propositions, we give the following local convergence theorem for the stochastic saddle-search algorithm with stochastic eigenvector search. 

\begin{theorem}[Local high-probability convergence]
    Suppose that the \Cref{step size sec 1}, \Cref{unbias}, \Cref{assumption: Hessian} hold, and \Cref{step size assumption for v} holds with $\Hvec=\nabla^2f(\x)$ for $\x\in \{\x,\Vert\x-\x^*\Vert\leqslant\delta\}$. Define $$Q=\max\left(1,\frac{1}{\min(\Delta^2/4,\mu^2)},\frac{1}{L^2}\right),\Delta=\min_{\Vert\x-\x^*\Vert\leqslant\delta}\min_{\lambda_i(\x)<\lambda_j(\x)<0} \lambda_j(\x)-\lambda_i(\x),$$ i.e., $\Delta$ is the minimal gap between distinct negative eigenvalues of the Hessian $\nabla^2 f(\x)$ within the neighborhood $\{\x,\Vert\x-\x^*\Vert\leqslant\delta\}$. Set a small enough tolerance $\epsilon_\vvec>0$ in \Cref{algorithm vector} and \Cref{algorithm}, which satisfies \eqref{smalll tolerance} and makes the error of the projection $\theta=k \bar{z}_k^2 d$ (\Cref{distance between projection}) satisfy $(1-\sqrt{\theta})\mu-L\theta-5L\sqrt{\theta}>0$ (in \Cref{descent with inexact eigenvector}). Next, define the attraction domain as $$U=\left\{\x,\Vert \x-\x^* \Vert \leqslant \min\left( \frac{(1-\sqrt{\theta})\mu-L\theta-5L\sqrt{\theta}}{M},\delta\right)\right\}.$$ Then, there exists a nonempty $U_0\subset U$ defined in \eqref{definition of U0}, and sufficiently large $\gamma,m>0$, such that, for $\x(n)$ generated by \Cref{algorithm} with $\x(0) \in U_0, \alpha(n)=\frac{\gamma}{n+m}$, we have the event $E^\infty = \{ \x(n) \in U, \forall n \}$ occurs with a large probability (\Cref{attraction domain}). Conditioned on $E^\infty$, $\x(n)$ almost surely achieves the tolerance $\Vert\nabla f(\x(n))\Vert_2^2<L^2\epsilon_\x$ in a finite number of iterations (\Cref{local a.s. convergence with inexact eigenvector} and \Cref{error of the inexact eigenvector}), with an $O(1/n)$ convergence rate in the sense of expectation (\Cref{convergence rate}). 
    \label{local convergence theorem}
\end{theorem}

To clarify the conditions given in the above theorem, we may simply interpret it as, if the initial condition $\x(0)$ is sufficiently close to the target index-$k$ saddle point $\x^*$, the step size (for updating $\x$) decay with a small initial step size, the tolerance $\epsilon_\vvec$ for the stochastic eigenvector search is small enough, then the index-$k$ stochastic saddle-search algorithm with stochastic eigenvector search has a large probability ($\mathbb{P}(E^\infty)=1-\epsilon_{E^\infty}\rightarrow1$ as $m\rightarrow \infty$) to achieve the tolerance ($\Vert\nabla f(\x(n))\Vert_2^2<L^2\epsilon_\x$) in a finite number of iterations.
\begin{remark}
    Even when the initial condition $\x(0)$ is sufficiently close to the saddle point, or even initialized exactly at the saddle point, we still may not have 
    the almost sure convergence (except for special cases where the basin of attraction of the saddle point covers the entire domain, e.g., in the case of a quadratic function, and except for special types of stochastic gradients, such as those with zero variance). The reason is that the presence of stochastic gradient noise always carries a positive probability of pushing the iterate out of the attraction domain of the saddle point.
\end{remark}

\section{Numerical experiment}\label{sec: numerical results}
In this section, we present a series of numerical examples to show the practical applicability, the convergence rate, and the advantages of the stochastic saddle-search algorithm developed in this work. Note that near a strict saddle point, we have
$$
    \mu^2\Vert \x(n)-\x^* \Vert^2_2\leqslant \Vert\nabla f(\x(n))\Vert^2_2=\Vert\nabla f(\x(n))-\nabla f(\x^*)\Vert^2_2\leqslant L^2\Vert \x(n)-\x^* \Vert^2_2.
$$
Thus, both $\Vert \x(n)-\x^* \Vert^2_2$ and $\Vert\nabla f(\x(n))\Vert^2_2$ can measure the distance between the iteration point and the target saddle point.
\subsection{Müller-Brown potential}\label{sec: MB}

This numerical example aims to illustrate the practicality of the stochastic saddle-search algorithm. The Müller-Brown (MB) potential is a two-dimensional function in theoretical chemistry that describes a system with multiple local minima and saddle points, serving as a standard benchmark to test saddle-search algorithms. The MB potential is given by \cite{bonfanti2017methods}
$$
    E(x,y)=\sum_{i=1}^4A_i e^{a_i(x-\bar{x}_i)^2+b_i(x-\bar{x}_i)(y-\bar{y}_i)+c_i(y-\bar{y}_i)^2},
$$
where $A=[-200,-100,-170,15]$, $a=[-1,-1,-6.5,0.7]$, $b=[0,0,11,0.6]$, $c=[-10,-10,-6.5,0.7]$, $\bar{x}=[1,0,-0.5,-1]$, and $\bar{y}=[0,0.5,1.5,1].$ We show the contour plot of the MB potential in \Cref{fig: MB} with two local minimizers (the reactant and product) and one index-1 saddle point (transition state) connecting two minimizers. In the typical setup for the saddle‑search algorithm, we start near a local minimizer to locate the adjacent index‑1 saddle point. Accordingly, we set the initial condition to $\x(0)=(-0.4,0.6)^\top$,  which lies near the bottom minimizer.

\begin{figure}
    \centering
    \includegraphics[width=0.99\linewidth]{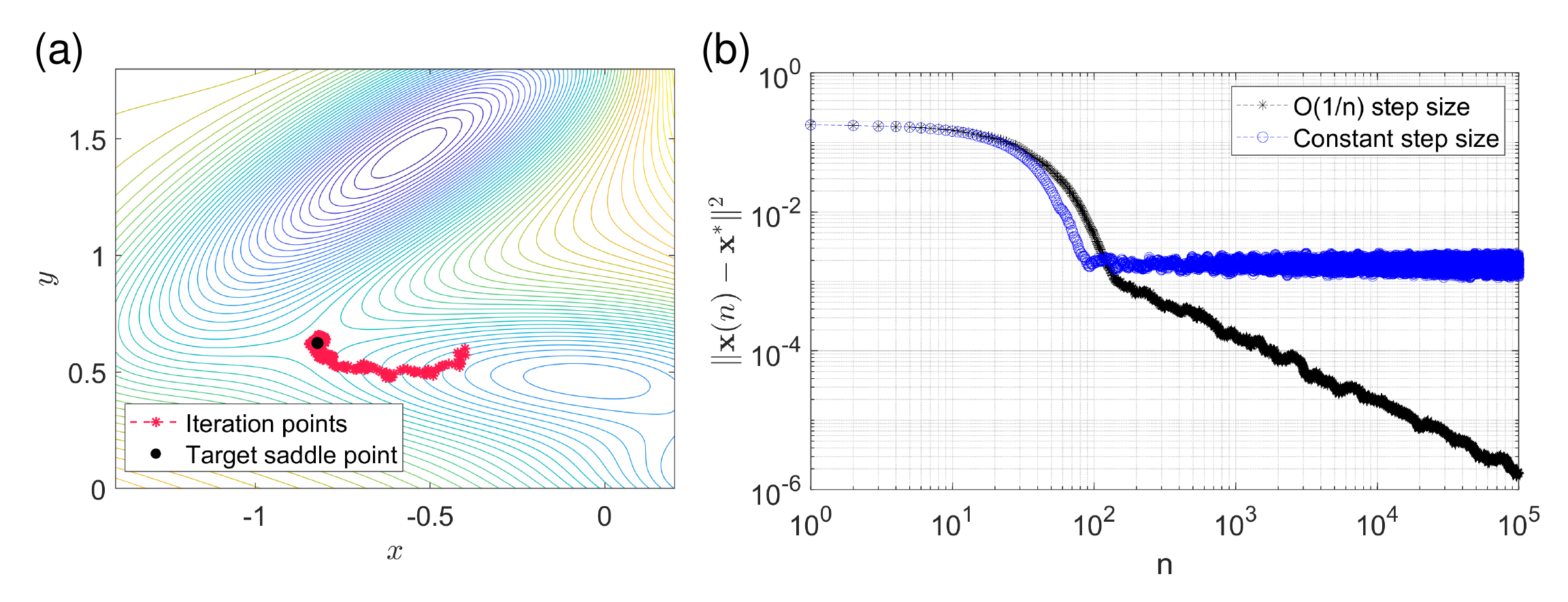}
    \caption{(a) Contour plot of MB potential and iteration points of one-time stochastic saddle-search algorithm with $\nabla f(\x;\omega)=\nabla f(\x)+100\xi, \xi\sim\mathcal{N}(\mathbf{0},\I_d)$. (b) Average error $\Vert \x(n)-\x^* \Vert^2$ over 100 runs of the stochastic saddle‑search algorithm with $\alpha(n)=0.01/(n+100)$ (black) and $\alpha(n)\equiv0.0001$ (blue).}
    \label{fig: MB}
\end{figure}

With a Gaussian‑perturbed stochastic derivative $\nabla f(\x;\omega)=\nabla f(\x)+100\xi, \xi\sim\mathcal{N}(\mathbf{0},\I_d)$, where $\xi$ is a $d$-dimensional standard Gaussian vector, we run the stochastic saddle‑search algorithm 100 times and compute the average error, as shown in \Cref{fig: MB}(b). The results indicate that a vanishing step‑size policy is crucial for the convergence. For $\alpha(n)=0.01/(n+100)$, the convergence rate near the target saddle point is $O(1/n)$: from iteration $10^2$ to $10^5$, the error decreases from $10^{-3}$ to $10^{-6}$. In contrast, the constant step size performs badly. With a constant step size, the error initially decreases because the true derivative dominates the approximation. However, as the iteration point approaches $\x^*$, the true derivative becomes small and the stochastic error dominates. Consequently, a constant step size leads to a relatively large near-stable‑state error, where the iterates oscillate around the target saddle point rather than converging to it.

\subsection{A butterfly energy landscape}

This numerical example aims to illustrate the capability of the stochastic saddle-search algorithm to escape a ``bad" region and transition into a ``good" or attractive region. If the initial condition is chosen away from the attraction domain of the saddle point, the iterates generated by the deterministic saddle dynamics may fail to converge to the target saddle point. In contrast, the stochastic saddle-search algorithm provides a possibility of circumventing this constraint. To give a clear perspective, we consider the following two-dimensional butterfly energy function \cite{su2025improved}:
$$
    E(x,y)=x^4-2x^2+y^4+y^2-1.5x^2y^2+x^2y-y^3.
$$

In \Cref{fig: butterfly}, we plot the trajectory of the deterministic saddle dynamics and stochastic saddle search applied to the butterfly energy function, respectively. One could observe that the trajectory of the deterministic saddle dynamics  (the blue line) could not cross the boundary of the region of attraction to reach the saddle point, which prevents the system from fully exploring the space. The stochastic saddle-search algorithm can address these limitations by introducing a stochastic error to escape from the ``bad" area (see the red line in \Cref{fig: butterfly}). This is analogous to the stochastic gradient descent, which can escape the attraction basin of a local minimizer and potentially converge to a better local or global minimizer \cite{kleinberg2018alternative, damian2021label}.

\begin{figure}
    \centering
    \includegraphics[width=0.55\linewidth]{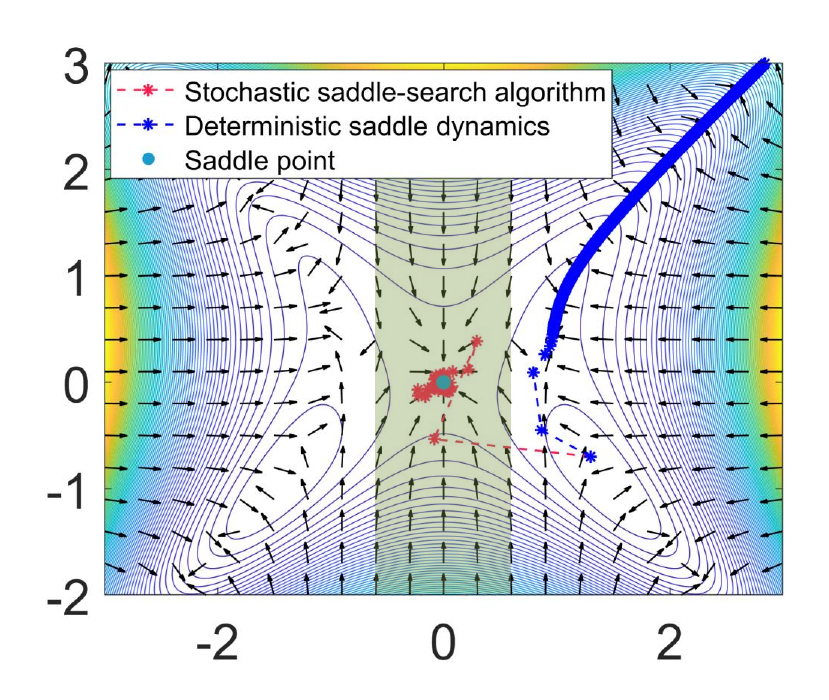}
    \caption{Contour plot of the butterfly energy landscape, together with trajectories of the stochastic saddle‑search algorithm ($\nabla f(\x;\omega)=\nabla f(\x)+\xi, \xi\sim\mathcal{N}(\mathbf{0},\I_d)$) and the deterministic saddle dynamics with $\alpha(n)=0.5/n$. The black quivers indicate the direction of the modified gradient with reflection in the eigenvector associated with the smallest eigenvalue, i.e., $-(\I_2-2 \vvec \vvec^\top)\nabla f(\x)$. The olive‑colored region roughly depicts the domain of attraction of the saddle point.}
    \label{fig: butterfly}
\end{figure}

\subsection{Loss landscape of a  neural network}

We implement the stochastic saddle-search algorithm on the loss landscape of the linear neural network to show that it also works for the degenerate high-index saddle point. Due to the non-convex nature of the loss function, multiple local minima and saddle points are the main concerns of the neural network optimization. Because of the overparameterization of neural networks, most critical points of the loss function are highly degenerate, i.e. the Hessian contains many zero eigenvalues. Let $\{(\x_i,\y_i)\}_{i=1}^{N}$ with $\x_i\in \mathbb{R}^{d_x}$ and $\y_i\in \mathbb{R}^{d_y}$ be the training data. We consider a fully-connected linear neural network $   Net(\W,\x)=\W_H\W_{H-1}\cdots \W_2 \W_1 
 \x$
 of depth $H$, with weight parameters $\W_{i+1}\in \mathbb{R}^{d_{i+1}\times d_i}$ for $d_0=d_x, d_H=d_y$. The corresponding empirical loss $f$ is defined by, for
$\mathbf{X}=[\x_1,\cdots,\x_N]$,  and $\mathbf{Y}=[\y_1,\cdots,\y_N]$,
$$    f(\W)=\sum_{i=1}^N\Vert Net(\W,\x_i)-\y_i \Vert_2^2=\Vert \W_H\W_{H-1}\cdots \W_2 \W_1\mathbf{X}-\mathbf{Y} \Vert_F^2,
$$
If $d_i = d_0, 1\leqslant i \leqslant H-1$, then $\W^*=[\W_1^*,\cdots,\W_H^*]$ is a saddle point where
$$
\W_1^* = \left[ \begin{array}{cc}
U_{\mathcal{S}}^\top \Sigma_{YX} \Sigma_{XX}^{-1} \\
0
\end{array} \right],  
\W_h^* = \I_{d_0} \text{ for } 2 \leqslant h \leqslant H - 1, \text{ and } 
\W_H^* = \left[ \begin{array}{cc}
\mathbf{U}_{\mathcal{S}} , \mathbf{0}
\end{array} \right], 
$$
and $\mathcal{S}$ is an index subset of $\{1,2,\ldots,r_{\max}\}$ for $r_{\max} = \min\{d_0,\ldots,d_H\}$, 
$\Sigma_{XX} = XX^\top$, $\Sigma_{YX} = YX^\top$, $\Sigma = \Sigma_{YX} \Sigma_{XX}^{-1} \Sigma_{YX}^\top$, and $U$ satisfies 
$\Sigma = U \Lambda U^\top$ with $\Lambda = \mathrm{diag}(\lambda_1, \ldots, \lambda_{d_y})$ \cite{achour2024loss,luo2025accelerated}. 

We set the depth $H = 5$, the input dimension $d_x = 10$, the output dimension $d_y = 4$, $d_i = 10$ for 
$1 \leqslant i \leqslant 4$, and the number of data points $N=100$ or $N=10000$. Data points $(\x_i, \y_i)$ are sampled from the normal distributions 
$\mathcal{N}(\mathbf{0}, \I_{d_x})$ and $\mathcal{N}(\mathbf{0}, \I_{d_y})$, respectively. The initial condition is 
$(\W_1(0), \cdots, \W_H(0)) = \W^* + (\mathbf{V}_1, \cdots, \mathbf{V}_H)$, where $(\mathbf{V}_1, \cdots, \mathbf{V}_H)$ is a random perturbation whose 
elements are drawn from $\mathcal{N}(0, \sigma_h^2)$ independently, with $\sigma_h = \| \W_h^* \|_F/(\sqrt{d_h - 1} d_h).$ Under the current setting, $\W^*$ 
is a degenerate saddle point with $16$ negative eigenvalues and several zero eigenvalues. We employ a mini-batch approach, where at each iteration $n$, the stochastic gradients are computed using a mini-batch loss function defined as
$$
f(\mathbf{W};\omega(n)) = \left\| \mathbf{W}_H \mathbf{W}_{H-1} \cdots \mathbf{W}_2 \mathbf{W}_1 \mathbf{X}_{I_n} - \mathbf{Y}_{I_n} \right\|_F^2,
$$
where $I_n \subset \{1, \dots, N\}$ is a uniformly sampled index set with the mini-batch size $|I_n| = 20$ (for $N = 100$) or $|I_n| = 1000$ (for $N = 10000$). The matrices $\mathbf{X}_{I_n}$ and $\mathbf{Y}_{I_n}$ denote the submatrices of $\mathbf{X}$ and $\mathbf{Y}$, respectively, formed by selecting the columns indexed by $I_n$. In \Cref{fig: NN}, we tested the stochastic saddle search algorithm on both relatively small and large datasets. In both cases, the stochastic algorithm successfully locates the target saddle point, with a convergence rate of $O(1/n)$.

\begin{figure}
    \centering
    \includegraphics[width=0.99\linewidth]{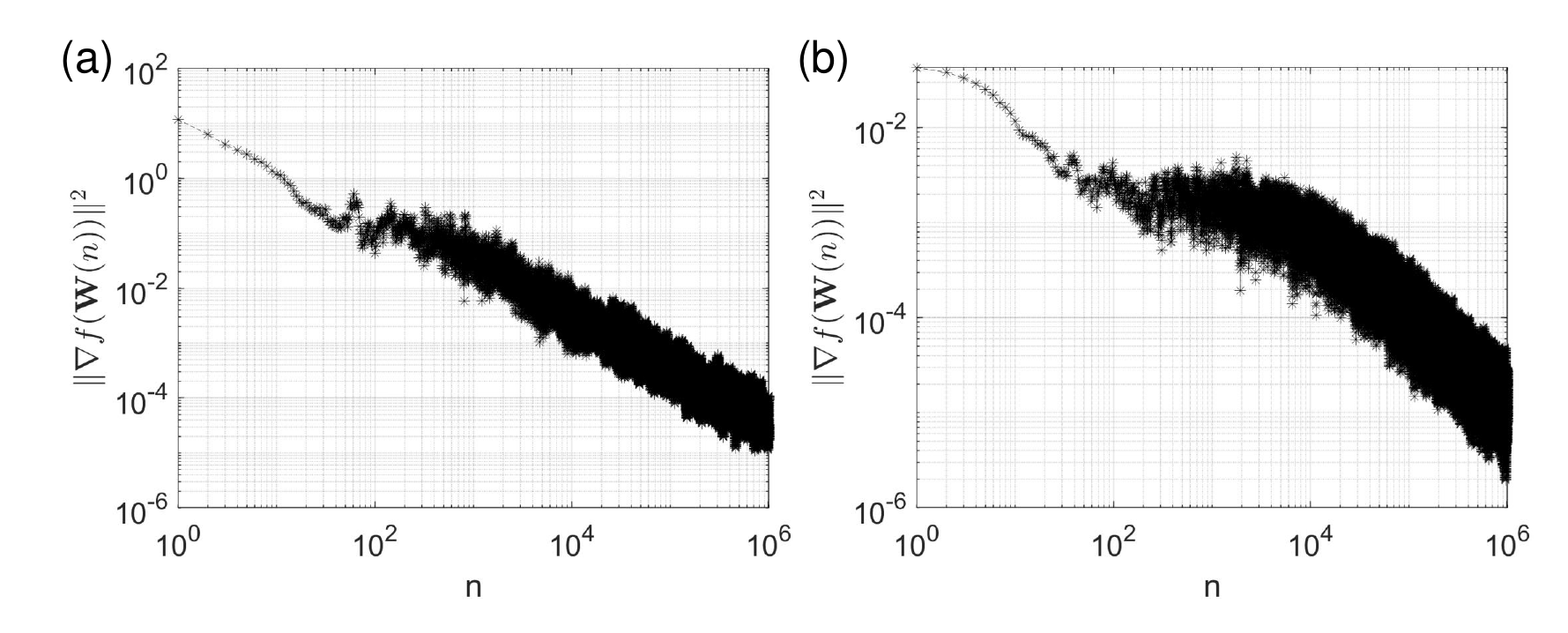}
    \caption{Plots of the squared gradient norm $\|\nabla f(\mathbf{W}(n))\|_2^2$ versus the iteration number $n$ with decay step size $\alpha(n)=-100/(n+10000)$, for the neural network loss function. (a) The dataset size is $N = 100$, with a mini-batch size $|I_n|=20$. (b) $N = 10000$, $|I_n|=1000$.
}
    \label{fig: NN}
\end{figure}

\subsection{Landau-de Gennes energy functional}\label{sec: Landau-de Gennes energy functional}
The last numerical example aims to demonstrate that the stochastic saddle-search algorithm is also applicable to,
via suitable finite-dimensional discretization, continuum energy functionals, such as the reduced Landau-de Gennes (LdG) energy functional \cite{gupta2005texture,brodin2010melting,muvsevivc2006two,Mu2008Self} for nematic liquid crystal (NLC). The reduced LdG model describes the director of the NLC state with a macroscopic $\mathbf{Q}$-tensor order parameter. The non-dimensionalized reduced LdG energy functional is given by \cite{wang2019order}
$$
  E[\mathbf{Q}(x,y)]: = \int_{\mathcal{D}} \dfrac{1}{2}|\nabla \mathbf{Q}(x,y)|^2+ \lambda^2\left(-\dfrac{B^2}{8C^2}\mathrm{tr}(\mathbf{Q}(x,y)^2)+\dfrac{1}{8}\left(\mathrm{tr}(\mathbf{Q}(x,y)^2\right)^2\right) \mathrm{d} S.
$$
where $\mathcal{D}=[-1,1]^2$ is non-dimensionalized domain. We set  $\lambda^2=15$, while the other parameters are kept the same as in \cite{shi2022nematic}. We use finite difference methods for spatial discretization with mesh size $\delta x$, and the function $\mathbf{Q}(x,y)$ is discretized into a $d$-dimensional vector $\mathbf{Q}$; the energy functional $E[\mathbf{Q}(x,y)]$ is discretized into a function $E(\mathbf{Q})$. The stochastic gradient is chosen as $\nabla E(\mathbf{Q}(n);\omega(n))$ = $P_{I_n} \nabla E(\mathbf{Q})$, where $I_n\subset \{1,\cdots,d\}$ with $|I_n|\approx d/10$ is a uniformly random index set sampled at iteration $n$, and $P_{I_n}$ zeroes out all coordinates outside $I_n$.

There are two stable diagonal (D) states in which the nematic director aligns with one of the square diagonals; the minimum-energy transition pathway between these dual global minima passes through the index-1 boundary-distortion (BD) saddle, as shown in \Cref{LdG}(a). We seek this index-1 BD with an initial condition near the D$_1$ state. In \Cref{LdG}(b), we plot the error $\Vert\nabla E(\mathbf{Q}(n))\Vert^2_2$ versus $n$. The iterates converge to the BD saddle, and we observe an empirical decay faster than $O(1/n)$: over $10^6$ iterations, $\Vert\nabla E(\mathbf{Q}(n))\Vert^2_2$ decreases from $10^{-1}$ to $10^{-12}$. This accelerated decay is actually induced by the variance reduction of the stochastic gradient. Note that the stochastic gradient adopted here satisfies $\mathbb{E}(\Vert \nabla E(\mathbf{Q}(n);\omega)-\nabla E(\mathbf{Q}(n)) \Vert^2_2) \leqslant \Vert \nabla E(\mathbf{Q}(n))\Vert^2_2$. As the iterate approaches the saddle, the true gradient vanishes, so the variance of the stochastic gradient diminishes. 

\begin{figure}
    \centering
    \includegraphics[width=0.9\linewidth]{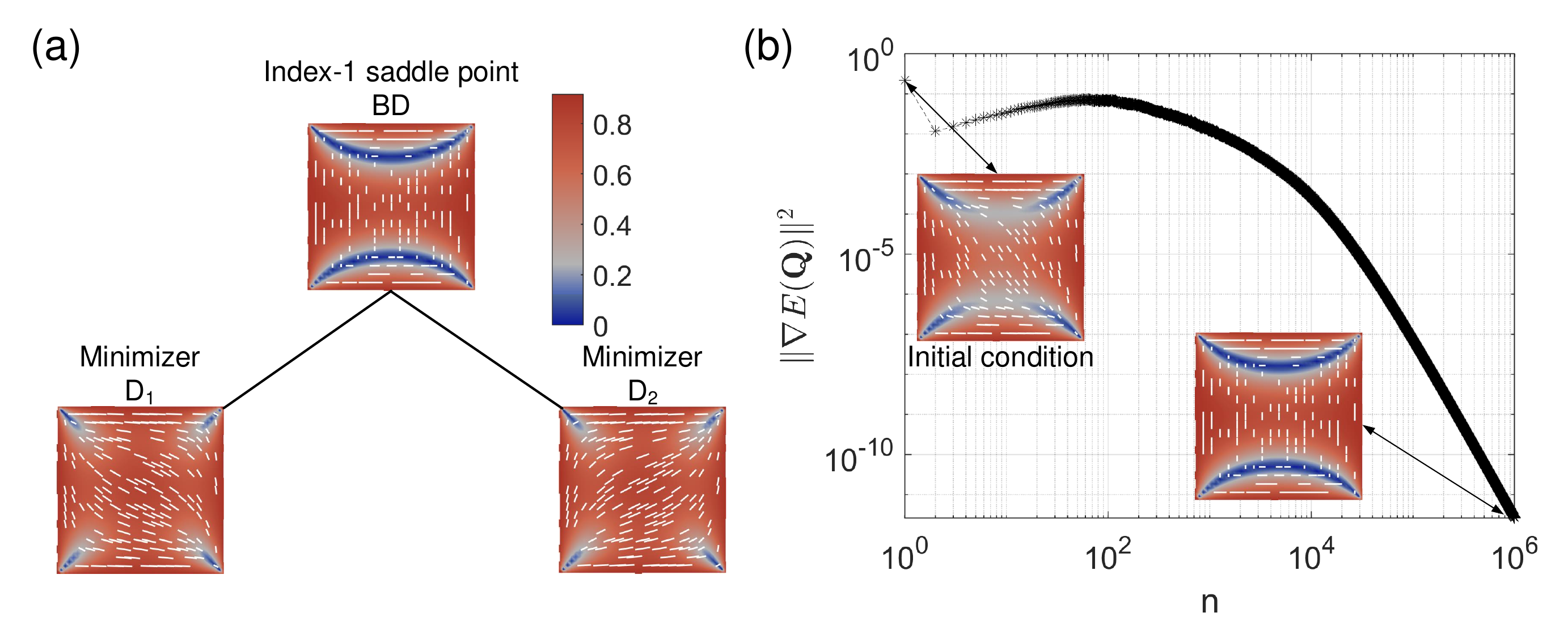}
    \caption{(a) Transition pathway between two dual stable states D$_1$ and D$_2$ passing through the index‑1 transition state BD. The color bar labels the order parameter $\sqrt{|\mathbf{Q}|^2/2}$ and the white lines represent the directors, i.e., the eigenvector corresponding to the largest eigenvalue of $\mathbf{Q}$. (b) The error $\lVert\nabla E(\mathbf{Q}(n))\rVert^2_2$ versus iteration count $n$ for the stochastic saddle‑search algorithm with step size $\alpha(n)=1/(n+10000)$ and an initial condition near the D$_1$ state.}
    \label{LdG}
\end{figure}

\section{Conclusion and discussion}\label{sec: conclusion}
This work is concerned with the development and numerical analysis of stochastic algorithms for the saddle point search. Leveraging global information about the objective function, we prove global almost sure convergence under certain conditions. These include a strongly convex-concave structure with a known unstable subspace (\Cref{proposition: converge with determined space}), and a uniformly bounded, slightly varying Hessian (\Cref{Global convergence with exact eigenvector} and \Cref{Global convergence with inexact eigenvector}). In scenarios where only local information near the saddle point is available, we establish the local convergence. We prove that the stochastic eigenvector-search algorithm almost surely identifies approximate unstable directions, with the error bounded by a tolerance parameter $\epsilon_{\mathbf{v}}$ (\Cref{distance between projection}). Using these approximate unstable directions, we construct the local high-probability convergence of the overall algorithm (\Cref{local convergence theorem}).

As the first attempt in the algorithmic development and systematic numerical analysis of stochastic saddle search with an unknown unstable space, the present work is still at the preliminary exploration stage. It opens up many promising directions to be pursued. For example,
the presence of nonlinear constraints increases the difficulty of saddle point search in many practical applications, such as the Oseen-Frank theory for nematic liquid crystals \cite{frank1958liquid} and Kohn-Sham density functional theory \cite{bickelhaupt2000kohn}. One potential avenue is to generalize the convergence results to the stochastic saddle-search algorithm on manifolds \cite{liu2023constrained, yin2020constrained}. Additional steps in the stochastic iteration may need to be introduced to deal with the retractions and vector transport on manifolds. Moreover, the current work primarily focuses on the convergence guarantees, with very limited study on the convergence rates and sampling efficiency. Using either a constant or adaptive step size can accelerate convergence to a neighborhood of the saddle point. However, due to the inherent stochastic variance in the algorithm, the iterates typically do not converge exactly to the saddle point with a constant step size (see \Cref{fig: MB}(b)). Therefore, how to select an appropriate step size to balance the convergence speed and the accuracy and how to apply variance reduction techniques
are also important research questions for many application problems.

\section*{Appendix: Details for the proof}
\subsection*{\Cref{proposition: converge with determined space}}
\begin{proof}

    Let $D(\xx,\xxx)=
    \frac{1}{2}(\Vert \xx-\xx^*\Vert^2_2+\Vert \xxx-\xxx^* \Vert^2_2)
    $ be the distance between $(\xx,\xxx)$ and $(\xx^*,\xxx^*)$. Due to the asymptotic convergence of the saddle dynamics established in \Cref{Lyapunov},
    we know that $D(\hat{\xx}(t_1),\hat{\xxx}(t_1))$ is small with a sufficiently large $t_1$. From \Cref{pseudo-trajectory}, we have that the difference between $D(\barxx(t+t_1),\barxxx(t+t_1))$ and $D(\hat{\xx}(t_1),\hat{\xxx}(t_1))$ is small with a sufficiently large $t$, which demonstrate $D(\barxx(t+t_1),\barxxx(t+t_1))$ is small enough with large enough $t$ and $t_1$. 
    
    Specifically, by \Cref{proposition: boundedness}, we may assume that the iterates stay in a bounded region $U_{C_1}=\{\x,\Vert\x-\x^*\Vert_2\leqslant C_1\}$, for some constant $C_1>0$. Let \( (\hat{\xx}(t), \hat{\xxx}(t)) \) be any solution of the saddle dynamics with initial condition \( (\hat{\xx}(0), \hat{\xxx}(0)) \in U_{C_1} \). By the asymptotic convergence of the saddle dynamics to the saddle point, that is, there exists a $T_1>0$ which is only dependent on $C_1$ and $\epsilon$, if $t_1>T_1$, then \( D(\xx(t_1), \xxx(t_1)) 
    \leqslant \frac{\epsilon}{2} \). Since \( (\barxxx(t), \barxxx(t)) \) is the asymptotic pseudo-trajectory of dynamics, we almost surely have:
\[
\lim_{t \to \infty} \sup_{0 \leqslant s \leqslant t_1} \|\barxx(t + s) - \hat{\xx}(s)\|_2 = 0, \lim_{t \to \infty} \sup_{0 \leqslant s \leqslant t_1} \|\barxxx(t + s) - \hat{\xxx}(s)\|_2 = 0,
\]
that is, there exists a $t_0(t_1(\epsilon), \epsilon) > 0$ such that for any \( \, 
t> t_0 \), we almost surely have 
$$
\begin{aligned}
 &\sup_{0 \leqslant s \leqslant t_1} \|\barxx(t + s) - \hat{\xx}(s)\|_2 < \min\left(\sqrt{\frac{\epsilon}{4}},\frac{\epsilon}{16C_1}\right),\\
 &\sup_{0 \leqslant s \leqslant t_1} \|\barxxx(t + s) - \hat{\xxx}(s)\|_2 <\min\left(\sqrt{\frac{\epsilon}{4}},\frac{\epsilon}{16C_1}\right),
\end{aligned}
$$
where \( (\hat{\xx}(s), \hat{\xxx}(s)) \) is the true trajectory of the saddle dynamics with initial conditions \( \hat{\xx}(0) = \barxx(t) \) and \( \hat{\xxx}(0) = \barxxx(t) \). Then we almost surely have that
$$
\begin{aligned}
D&(\barxx(t+t_1), \barxxx(t + t_1)) \\
= &D(\xx(t_1), \xxx(t_1))+\left(\Vert \xx(t_1)- \barxx(t+t_1) \Vert^2_2+\Vert \xxx(t_1)- \barxxx(t+t_1) \Vert^2_2\right)/2\\
&+\langle \barxx(t+t_1)-\xx(t_1),\xx(t_1)-\xx^* \rangle + \langle \barxxx(t+t_1)-\xxx(t_1),\xxx(t_1)-\xxx^* \rangle \\
< &\frac{\epsilon}{2}+\min\left(\sqrt{\frac{\epsilon}{4}},\frac{\epsilon}{16C_1}\right)^2+
4C_1\min\left(\sqrt{\frac{\epsilon}{4}},\frac{\epsilon}{16C_1}\right)\leqslant \epsilon.
\end{aligned}
$$
That is, for any  
$\epsilon>0$, there exists $T_0=t_1(\epsilon)+t_0(\epsilon)$ such that, if $t>T_0$, we almost surely have $D(\barxx(t), \barxxx(t))< \epsilon$, which demonstrates the almost sure convergence of $(\barxx(t),\barxxx(t))$ to the saddle point $(\xx^*,\xxx^*)$ as $t\rightarrow\infty$. 
\end{proof}

\subsection*{\Cref{lemma: An and Bn}}

\begin{proof}

From the fact that $B(\x) \geqslant 0$, $A(\x)$ is bounded from below, and using \eqref{key equation for convergence}, we conclude that the process
$$
A(\x(n)) + \sum_{i=n}^\infty C_3 \alpha(i)^2
$$
is a supermartingale bounded from below. Therefore, it converges almost surely from Doob’s martingale convergence theorem. This implies that $A(\x(n))$ itself converges almost surely (because $\sum_{i=0}^\infty C_3\alpha(i)^2<\infty$). Furthermore, from \eqref{key equation for convergence}, we see that $A(\x(n))$ is a stochastic process with a bounded positive variation, that is
$$
\sum_{n=0}^\infty \max(\mathbb{E}(A(\x(n+1))-A(\x(n))|\mathcal{F}(n)),0)<\infty.
$$
Since $A(\x)$ is also bounded from below, $A(\x(n))$ also has a bounded negative variation. 
Hence, $A(\x(n))$ has a bounded total variation, that is
$$
\sum_{n=0}^\infty \left| \mathbb{E}(A(\x(n+1)) - A(\x(n)) \mid \mathcal{F}(n)) \right| < \infty.
$$
It follows from summing both sides of \eqref{key equation for convergence} that
$$
C_2\sum_{n=0}^\infty\alpha(n)B(\x(n))\leqslant \sum_{n=0}^\infty C_3\alpha(n)^2 +
\sum_{n=0}^\infty \left| \mathbb{E}(A(\x(n+1)) - A(\x(n)) |\mathcal{F}(n)) \right|<\infty.
$$
Finally, since $\sum_{n=0}^\infty \alpha(n) B(\x(n)) < \infty$, $B(\x(n))$ converges almost surely as $n\rightarrow \infty$, and $\sum_{n=0}^\infty \alpha(n) = \infty$, we conclude that $B(\x(n)) \to 0$ almost surely.
\end{proof}

\subsection*{\Cref{a.s. convergence of the eigenvector searching}}

\begin{proof}
Define Rayleigh Quotient as $A(\vvec)=\vvec^T\Hvec\vvec$, then 
    $$
    \begin{aligned}
        A&(\vvec(n+1))-A(\vvec(n))
        \\
        =&\vvec(n+1)^\top\Hvec\vvec(n+1)-\vvec(n)^\top\Hvec\vvec(n)\\
        &-\underbrace{\langle(\I-\vvec(n) \vvec(n)^\top) \Hvec(\omega(n))\vvec(n),\vvec(n)\vvec(n)^\top\Hvec \vvec(n)\rangle}_{=0}\\
        \leqslant& -2\alpha(n)\langle (\I-\vvec(n) \vvec(n)^\top) \Hvec(\omega(n))\vvec(n),(\I-\vvec(n) \vvec(n)^\top) \Hvec \vvec(n) \rangle\\
        &+M_2 \alpha(n)^2\Vert\Hvec(\omega(n))\Vert^2_2,
        \end{aligned}
    $$
where $M_2$ is a positive constant only dependent on $G_1$. Taking expectations on both sides, we have
\begin{equation}
\begin{aligned}
    &\mathbb{E}(A(\vvec(n+1))-A(\vvec(n))|\mathcal{F}(n))\\
    &\leqslant -2\alpha(n)\Vert (\I-\vvec(n) \vvec(n)^\top) \Hvec \vvec(n) \Vert^2_2+M_2 \alpha(n)^2\mathbb{E}(\Vert\Hvec(\omega(n))\Vert_2^2)\\
    &\leqslant -2\alpha(n)\Vert (\I-\vvec(n) \vvec(n)^\top) \Hvec \vvec(n) \Vert^2_2+M_2 \alpha(n)^2\sigma^2.
\end{aligned}
\label{positive variation}
\end{equation}
Setting $B(\vvec)=\Vert (\I-\vvec \vvec^\top) \Hvec \vvec \Vert^2_2$ and applying \Cref{lemma: An and Bn} with \eqref{positive variation}, we have that $\sum_{n=0}^\infty \alpha(n)B(\vvec(n))$ almost surely converges. 

Now we prove that $B(\vvec(n))$ almost surely converges. Note that
$$
\begin{aligned}
    &\mathbb{E}(B(\vvec(n+1))-B(\vvec(n))|\mathcal{F}(n))
    \\
    &=-\alpha(n)\langle (\I-\vvec(n) \vvec(n)^\top) \Hvec \vvec(n), 2\mathbf{W}(\Hvec,\vvec(n))(\I-\vvec(n) \vvec(n)^\top) \Hvec \vvec(n)\rangle\\
    &\quad \quad+O(\alpha(n)^2),
    \end{aligned}
$$
where 
$$
     \mathbf{W}(\Hvec,\vvec(n)) =\Hvec -\vvec(n)^\top\Hvec \vvec(n)\I, \Vert\mathbf{W}(\Hvec,\vvec(n))\Vert_2\leqslant 2\Vert \Hvec \Vert_2,
$$
and thus 
$$
    |\mathbb{E}(B(\vvec(n+1))-B(\vvec(n))|\mathcal{F}(n))|
    \leqslant 4\alpha(n)\Vert\Hvec \Vert_2B(\vvec(n))+O(\alpha(n)^2),
$$
which means that the sequence $\{B(\vvec(n))\}$ has a bounded variation and the sequence converges almost surely. Thus, applying \Cref{lemma: An and Bn}, we have that $B(\vvec(n))$ almost surely converges to $0$. Since $\vvec(n)$ is always a unit vector, there exists a subsequence $\vvec(n_i)$ that converges to $\vvec(\infty)$, which almost surely satisfies
$B(\vvec(\infty))=\Vert (\I-\vvec(\infty) \vvec(\infty)^\top) \Hvec\vvec(\infty) \Vert^2_2=0$, that is
$$
    \Hvec \vvec(\infty)=(\vvec(\infty)^\top \Hvec \vvec(\infty))\vvec(\infty).
$$Thus, the subsequence $\vvec(n_i)$ almost surely converges to one of the eigenvectors of $\Hvec$. 
\end{proof}

\subsection*{\Cref{error of the inexact eigenvector}}

\begin{proof}
    If $\vvec_i,i=1,\cdots,\bar{k}-1$ are orthogonal vectors in $\mathbb{S}^{d-1}$ that are close to the exact eigenvectors of $\Hvec$ with a small error, we may assume that the error can be quantified by a small positive parameter $\theta$, i.e. 
    \begin{equation*}
    \vvec_i=\sum_{j=1}^d \alpha_{ij} \vvec^*_j, \alpha_{ii}^2>1-\theta,\;\; \sum_{j=1,j\neq i}^{j=d}\alpha_{ij}^2<\theta,\quad i=1,\cdots,\bar{k}-1,
    \end{equation*}
where $\vvec^*_j$ is the eigenvector of $\Hvec$ with $j$-th smallest eigenvalue. Define $\mathbf{r}_i=\vvec_i-\vvec_i^*$ (assume that $\vvec_i$ and $\vvec_i^*$ form an acute angle without loss of generality), then $\Vert \mathbf{r}_i \Vert<\sqrt{\theta}$, and  
\begin{equation}\label{eq: estimate the error}
    \begin{aligned}
        \Hvec \vvec_\infty-(\vvec_\infty^T\Hvec\vvec_\infty)\vvec_\infty&=\sum_{i=1}^{\bar{k}-1}(\vvec_i^\top \Hvec\vvec_\infty)\vvec_i\\
        &=\sum_{i=1}^{\bar{k}-1}(\lambda_i\vvec_i^T \vvec_\infty+\mathbf{r}_i^\top\Hvec)\vvec_i=\sum_{i=1}^{\bar{k}-1}(\Hvec\mathbf{r}_i)^\top\vvec_i.
        \end{aligned}
\end{equation}
Thus, 
\begin{equation}\label{eq:error}
\Vert\Hvec \vvec_\infty-(\vvec_\infty^T\Hvec\vvec_\infty)\vvec_\infty \Vert_2^2\leqslant (\bar{k}-1)^2L^2\theta\leqslant \bar{k}^2L^2\theta=z_{\bar{k}}^2 d.
\end{equation}
Assume $\vvec_\infty=\sum_{i=0}^d \alpha_{\infty,i}\vvec_i^*$ with $ \sum_{i=1}^d \alpha_{\infty ,i}^2=1$, we have  $$\sum_{i=1}^d(\lambda_i-\vvec_\infty^T\Hvec\vvec_\infty)^2\alpha_{\infty ,i}^2\leqslant \bar{k}^2L^2\theta= z_{\bar{k}}^2 d,$$ 
which yields $\min_{1\leqslant i\leqslant d}(\lambda_i-\vvec_\infty^T\Hvec\vvec_\infty)^2 \leqslant z_{\bar{k}}^2.$ 

Set $i^*=\min(\mathrm{argmin}_{1\leqslant i\leqslant d}(\lambda_i-\vvec_\infty^T\Hvec\vvec_\infty)^2)$, $\Delta=\min_{\lambda_i<\lambda_j<0} \lambda_j-\lambda_i$, then $$(\lambda_i-\vvec_\infty^T\Hvec\vvec_\infty)^2\geqslant\min\left(\left(\Delta-z_{\bar{k}}\right)^2,\left(2\mu-z_{\bar{k}}\right)^2\right),\lambda_i\neq \lambda_{i^*}.$$ 
Together with $\sum_{i=1}^d(\lambda_i-\vvec_\infty^T\Hvec\vvec_\infty)^2\alpha_{\infty ,i}^2\leqslant z_{\bar{k}}^2 d$, we have 
$$
\sum_{i=1,\lambda_i\neq \lambda_i^*}^d\alpha_{\infty,i}^2\leqslant \frac{z_{\bar{k}}^2 d}{\min\left(\left(\Delta-z_{\bar{k}}\right)^2,\left(2\mu-z_{\bar{k}}\right)^2\right)},
$$
which demonstrates that $\vvec_\infty$ is close to the eigenspace of $\lambda_{i^*}$. Specifically, define $V_{i^*}=[\vvec^*_{i^*},\cdots,\vvec^*_{i^*+n(\lambda^*)-1}]$, where each column is an eigenvector corresponding to the eigenvalue $\lambda_{i^*}$, and $n(\lambda^*)$ denotes the algebraic multiplicity of the eigenvalue $\lambda^*$, we have
\begin{equation}
1-\Vert \vvec_{\infty}^\top V_{i^*}\Vert^2_2 =\sum_{i=1,\lambda_i \neq \lambda_{i^*}}^d\alpha_{\infty,i}^2\leqslant \frac{z_{\bar{k}}^2 d}{\min\left(\left(\Delta-z_{\bar{k}}\right)^2,\left(2\mu-z_{\bar{k}}\right)^2\right)}.
 \label{error between unstable direction}
\end{equation}
If the algorithm stops with $$\left \Vert\left(\I-\vvec(n_{end})\vvec(n_{end})^\top-\sum_{i=1}^{\bar{k}-1}\vvec_i\vvec_i^\top\right)\Hvec\vvec(n_{end})\right\Vert^2_2< L^2\theta$$ at the $n_{end}$-th step, compared with \eqref{eq: estimate the error} for $\vvec_\infty$, there is another error $\Vert\bar{\mathbf{r}}\Vert^2_2\leqslant L^2\theta$ induced by the tolerance, 
    $$
    \begin{aligned}
        &\Hvec \vvec(n_{end})-(\vvec(n_{end})^T\Hvec\vvec(n_{end}))\vvec(n_{end})=\sum_{i=1}^{\bar{k}-1}(\vvec_i^\top \Hvec \vvec_\infty)\vvec_i+\bar{\mathbf{r}}\\
        &=\sum_{i=1}^{\bar{k}-1}(\lambda_i\vvec_i^T \vvec(n_{end})+\mathbf{r}_i^\top\Hvec)\vvec_i+\bar{\mathbf{r}}=\sum_{i=1}^{\bar{k}-1}(\Hvec\mathbf{r}_i)^\top\vvec_i+\bar{\mathbf{r}}.
        \end{aligned}
    $$
and $\vvec(n_{end})$ also satisfies 
$\Vert\Hvec \vvec(n_{end})-(\vvec(n_{end})^T\Hvec\vvec(n_{end}))\vvec(n_{end}) \Vert_2^2\leqslant z_{\bar{k}}^2 d$ as in \eqref{eq:error}. Following the same steps, \eqref{error between unstable direction} also holds for $\vvec(n_{end})$.
\end{proof}

\subsection*{\Cref{attraction domain}}

\begin{proof}
   Since \Cref{descent with inexact eigenvector} is analogous to \cite[Proposition D.1]{mertikopoulos2020almost}, which constitutes a key step in establishing the positive probability of the event $E^\infty$, the proof proceeds along the same lines as \cite[Theorem 4]{mertikopoulos2020almost}. Thus, we omit the details and give the main steps. The only places different from the original proof are: (i) the gradient term appears as $P_{\tilde{V}}\nabla f$ rather than $\nabla f$, which does not affect its $L^2$ norm; and (ii) the upper bound of $\|\nabla f(\x)\|_2$, $G = \max_{\x \in U}\|\nabla f(\x)\|_2$, is not global but restricted to $\x \in U$, which does not alter the proof steps.
   
   By \Cref{descent with inexact eigenvector}, the distance between the iteration point and saddle point grows at most by $\alpha(n)\psi(n)+\frac{1}{2}\alpha(n)^2\Vert \nabla f(\x(n);\omega(n)) \Vert^2_2$ at each step. Define the cumulative mean square error $$R^n=\left(\sum_{i=0}^n \alpha(i)\psi(i)\right)^2+\sum_{i=0}^n\frac{1}{2}\alpha(i)^2\Vert \nabla f(\x(i);\omega(i)) \Vert^2_2.$$
Set $\epsilon$ as
$$ 4\epsilon+2\sqrt{\epsilon}=\min\left( \frac{(1-\sqrt{\theta})\mu-L\theta-5L\sqrt{\theta}}{M},\delta\right)^2, 
$$
and define $U_0\subset U$ by 
$$
   U_0=\left\{\x,\Vert\x-\x^*\Vert_2\leqslant\sqrt{2\epsilon}\right\}.
$$
Thus, when the initial condition $\x(0)\in U_0$, and the umulative mean
square error $R^n<\epsilon$, we have that $$
\begin{aligned}
\Vert \x(n+1)-\x^*\Vert^2_2&\leqslant \Vert \x(0)-\x^* \Vert_2^2+2\sum_{i=0}^n \alpha(i)\psi(i)+\sum_{i=0}^n\alpha(i)^2\Vert \nabla f(\x(i);\omega(i)) \Vert^2_2\\
&\leqslant 2\epsilon+2\sqrt{\epsilon}+2\epsilon=4\epsilon+2\sqrt{\epsilon},
\end{aligned}
$$
i.e., $\x(n+1)\in U.$
Define the small error event $J^n=\{R_k\leqslant\epsilon,k=1,\cdots,n\}$, and the boundedness event $
        E^n = \{ \x(i) \in U, i=1,2,\cdots,n \}.$
By \cite[Lemma D.2.]{mertikopoulos2020almost}, we have $E^{n+1}\subseteq E^n,J^{n+1}\subseteq J^n,J^{n-1}\subseteq E^n$, which means that a small error ensures the occurrence of the boundedness event, and 

$$
\epsilon \mathbb{P}(J^{n}\setminus J^{n+1})\leqslant \mathbb{E}(R^{n}\mathbf{1}_{J^{n-1}})-\mathbb{E}(R^{n+1}\mathbf{1}_{J^{n}})+(G^2+(1+4\epsilon+2\sqrt{\epsilon})\sigma^2)\alpha(n+1)^2, 
$$
with $R^0 \mathbf{1}_{J^{-1}}=0$. By telescoping the above inequality, we get,
$$
\epsilon\mathbb{P}\left(\bigcup_{i=0}^n (J^{i}\setminus J^{i+1})\right)=\epsilon\sum_{i=0}^n\mathbb{P}(J^{i}\setminus J^{i+1})\leqslant \sum_{n=0}^{\infty}(G^2+(1+4\epsilon+2\sqrt{\epsilon})\sigma^2)\alpha(n)^2.
$$
Then,
$$
\begin{aligned}
\mathbb{P}(E^\infty)&=\inf_n \mathbb{P}(E^{n+1})\geqslant \inf_n \mathbb{P}(J^{n})=\inf_n \mathbb{P}\left(\bigcap_{i=0}^n (J^{i}\setminus J^{i+1})^c\right)\\
&=\inf_n \mathbb{P}\left(\Omega\setminus\bigcup_{i=0}^n (J^{i}\setminus J^{i+1})\right)\geqslant1 -\underbrace{\left(G^2+(1+4\epsilon+2\sqrt{\epsilon}\right)\sigma^2)\sum_{n=0}^\infty \frac{\alpha(n)^2}{\epsilon}}_{\epsilon_{E^\infty}}.
\end{aligned}
$$
Thus, if  $\sum_{n=0}^\infty \alpha(n)^2$ is small enough such that $1-\epsilon_{E^\infty}>0$ for the
$\epsilon_{E^\infty}$ defined on the right-hand side of the above inequality
, and the initial distance is not too large (i.e., $\x(0)\in U_0$), then $E^\infty$ occurs with a positive probability.
    
In a particular case, if the stochastic saddle-search algorithm is run with a step-size schedule of the form $\alpha(n)=\frac{\gamma}{(n+m)^p}$ with $p\in(1/2,1]$. As $m\rightarrow \infty$, $\sum_{n=0}^\infty \alpha(n)^2=\sum_{n=0}^\infty \frac{\gamma^2}{(n+m)^{2p}} \rightarrow 0$, and $1-\epsilon_{E^\infty}\rightarrow 1$, which means, with a large $m$, the event $E^\infty$ occurs with a large probability.
\end{proof}

\subsection*{\Cref{convergence rate}}

\begin{proof}
As in the case of SGD \cite{chung1954stochastic, mertikopoulos2020almost,polyak1992acceleration}, the proof is established on \Cref{descent with inexact eigenvector} and \Cref{attraction domain}. Define the events $
        E^n = \{ \x(i) \in U, i=1,2,\cdots,n \}$ and corresponding indicator function $\mathbf{1}_{E^n}$. From \Cref{descent with inexact eigenvector}, we have
    $$
    \begin{aligned}
        &\Vert \x(n+1)-\x^* \Vert^2_2\mathbf {1}_{E^n} \leqslant(1-[(1-\sqrt{\theta})\mu-L\theta-5L\sqrt{\theta}]\alpha(n))\Vert \x(n)-\x^* \Vert^2_2\mathbf {1}_{E^n}\\
        &\quad \quad+\alpha(n)\psi(n)\mathbf {1}_{E^n}+\alpha(n)^2\Vert \nabla f(\x(n);\omega(n)) \Vert^2_2\mathbf {1}_{E^n},
        \end{aligned}
    $$
and
    $$
    \begin{aligned}
    &    \mathbb{E}(\Vert \x(n+1)-\x^* \Vert^2_2\mathbf {1}_{E^{n+1}})
    \leqslant \mathbb{E}(\Vert \x(n+1)-\x^* \Vert^2_2\mathbf {1}_{E^n})\\
    & \qquad \leqslant (1-\underbrace{[(1-\sqrt{\theta})\mu-L\theta-5L\sqrt{\theta}]}_{C(\theta)}\alpha(n))\mathbb{E}(\Vert \x(n)-\x^* \Vert^2_2\mathbf {1}_{E^n})+2\alpha(n)^2(G^2+\sigma^2).
    \end{aligned}
    $$
    We proceed by induction to show that $\mathbb{E}(\Vert \x(n+1)-\x^* \Vert^2_2\mathbf {1}_{E^{n+1}})\leqslant \frac{\bar{M}}{n+m}$ with large enough $\gamma,m$. Clearly, the estimate holds for $n=0$ with a large enough $m$. Suppose that it holds for some $n\geqslant 0$, with $\gamma>1/C(\theta)$, and $\bar{M}=\frac{2\gamma^2(G^2+\sigma^2)}{C(\theta)\gamma-1}$, we have that
$$
    \begin{aligned}
        &\mathbb{E}(\Vert \x(n+1)-\x^* \Vert^2_2\mathbf {1}_{E^{n+1}})\\
    &\leqslant (1-C(\theta)\alpha(n))\mathbb{E}(\Vert \x(n)-\x^* \Vert^2_2\mathbf {1}_{E^n})+2\alpha(n)^2(G^2+\sigma^2)\\
    &\leqslant \left(1-\frac{C(\theta)\gamma}{n+m}\right)\frac{\bar{M}}{n+m}+\frac{2\gamma^2(G^2+\sigma^2)}{(n+m)^2}\\
    &\leqslant \frac{(n+m-1)\bar{M}}{(n+m)^2}-\frac{(C(\theta)\gamma-1)\bar{M}}{(n+m)^2}+\frac{2\gamma^2(G^2+\sigma^2)}{(n+m)^2}\\
    &\leqslant \frac{(n+m-1)\bar{M}}{(n+m)^2}=\frac{((n+m)^2-1)\bar{M}}{(n+m)^2(n+m+1)}\leqslant \frac{\bar{M}}{n+m+1}.
    \end{aligned}
$$
Then
$$
    \mathbb{E}(\Vert \x(n)-\x^* \Vert^2_2|E^\infty)=\frac{\mathbb{E}(\Vert \x(n)-\x^* \Vert^2\mathbf {1}_{E^\infty})}{\mathbb{P}(E^\infty)}\leqslant \frac{\bar{M}}{(n+m)\mathbb{P}(E^\infty)}.$$
Specifically, when $\gamma=\frac{2}{C(\theta)},$ $\bar{M}$ achieves its minimum, given by $ \frac{8(G^2+\sigma^2)}{C(\theta)^2},$ which intuitively demonstrates that large stochastic gradient errors (reflected in a large 
$\sigma^2$) and large errors in estimating the unstable subspace (reflected in a small $C(\theta)$) slow down the convergence of the iterations.
\end{proof}

\bibliographystyle{siamplain}
\bibliography{references}

\end{document}